\documentclass[EJP]{ejpecp} % replace ECP by EJP if needed.

\usepackage{lineno,hyperref}
\usepackage{amsmath}
\allowdisplaybreaks[4]
\usepackage{amssymb}
\usepackage{mathrsfs}
\usepackage{amsthm}
\usepackage{bm}
\usepackage{graphicx}
\usepackage{booktabs}
\usepackage{float}
\usepackage{subfig}
\usepackage{geometry}
\usepackage{bbm}
\usepackage{color}
\usepackage{pgf}
\definecolor{Green}{RGB}{0,128,0}

\SHORTTITLE{Hitting properties for FKE with time-fractional noise}

\TITLE{Hitting properties of generalized fractional kinetic equation with time-fractional noise\thanks{This work is supported by National key R\&D Program of China under Grant (No. 2020YFA0713701) and National Natural Science Foundation of China (No. 11971470, No. 12031020, No. 12171047).}} % \thanks is optional. Insert line breaks with \\

\AUTHORS{%
   Derui~Sheng\footnote{Academy of Mathematics and Systems Science, Chinese Academy of Sciences; School of Mathematical Sciences, University of Chinese Academy of Sciences, China. \EMAIL{sdr@lsec.cc.ac.cn} }
  \and %% remove this line and below if single author
  Tau~Zhou\footnote{Academy of Mathematics and Systems Science, Chinese Academy of Sciences; School of Mathematical Sciences, University of Chinese Academy of Sciences, China. \BEMAIL{zt@lsec.cc.ac.cn} ~~Corresponding author.}}%AUTHORS
 % Separate items with ;

\AMSSUBJ{60J45, 60G17, 60G22, 35R11} % Edit. Separate items with ;
\SUBMITTED{January 2, 2013} % Edit.
\ACCEPTED{December 13, 2014} % Edit.

\DOI{10.1214/YY-TN}

\ABSTRACT{This paper studies hitting properties for the system of generalized fractional kinetic equations driven by Gaussian noise fractional in time and white or colored in space. We derive the mean square modulus of continuity and some second order properties of the solution. These are applied to deduce lower and upper bounds for probabilities that the path process hits bounded Borel sets in terms of the $\mathfrak{g}_q$-capacity and $g_q$-Hausdorff measure, respectively, which yield the critical dimension for  hitting points. Further, based on some fine analysis of the harmonizable representation for the solution, we prove that all points are polar in the critical dimension. This provides strong evidence for the conjecture raised in Hinojosa-Calleja and Sanz-Solé [Stoch PDE: Anal Comp (2022). https://doi.org/10.1007/s40072-021-00234-6].}

\newcommand{\ud}{\mathrm d}
\newcommand{\ee}{\mathbb{E}}
\newcommand{\bp}{\mathbb{P}}
\newcommand{\br}{\mathbb{R}}

\newcommand{\ri}{\mathrm{i}}
\begin{document}
\section{Introduction}
In this paper, we consider the following generalized fractional kinetic equation
\begin{align}\label{FKE} 
\begin{cases}
\partial_{t} u(t, x)+(I-\Delta)^{\alpha/2}(-\Delta)^{\gamma/2} u(t, x)= \dot{W}^{H,\beta}(t, x), \quad (t,x) \in(0,T]\times\br^d,\\
u(0,x)=u_0(x), \quad x\in \br^d,
\end{cases}
\end{align}
where $u_0\in L^1(\br^d;\br)$, $T>0,\alpha \ge 0, \gamma>0$ and the noise $\dot{W}^{H,\beta}:=\frac{\partial^{d+1}W^{H,\beta}}{\partial t\partial x_1\dots\partial x_d}$ can be understood as a generalized Gaussian field which is fractional in time with a Hurst index $H\in(1/2,1)$ and white or colored in space with a scaling parameter $\beta\in(0,d]$ (see Section \ref{section2}
 for more details). The operators $(I-\Delta)^{\alpha/2}$ and $(-\Delta)^{\gamma/2}$ are respectively interpreted as inverses of the Bessel potential and the Riesz potential, whose composition is a natural mathematical object to describe the fractal phenomena and long-range dependence which display in many fields of application (see \cite{MR1808915} and references therein). 
This type of fractional kinetic equation plays an instrumental role in modeling various physical phenomena, such as the diffusion in porous media with fractal geometry, kinematics in viscoelastic media, propagation of seismic waves, etc (see \cite{MR1859007}), and we refer to for instance \cite{MR1808915,MR2144558,MR2571321,MR3999099,MR3374870} for its theoretic research. Moreover, \eqref{FKE} includes many special examples such as the stochastic heat equation ($\alpha = 0, \gamma = 2$) and the stochastic biharmonic heat equation ($\alpha = 0,\gamma=4$).

The goal of the present paper is to study hitting properties of the $\br^n$-valued random field $\hat u$ with components being independent copies of the solution $u$ to \eqref{FKE} which is well-posed if and only if $\alpha_1:=(\alpha+\gamma)H+\beta/2>0$. To be specific,
given $t_0\in(0,T]$ and $M>0$, we are interested in the hitting probability $$\bp(\hat u([t_0,T]\times[-M,M]^d) \cap A \neq \emptyset)$$ for a Borel set  $A \subseteq \br^n$. In particular, if $\bp(\hat u([t_0,T]\times[-M,M]^d) \cap A \neq \emptyset)=0$,  then $A$ is called polar for $\hat u$; otherwise, $A$ is called non-polar. A point $z \in \br^n$ is called polar for $\hat u$ if the singleton $\{z\}$ is polar. 
Hitting probabilities of solutions to systems of stochastic partial differential equations have been extensively studied in \cite{MR1902843,Hinojosa,MR2073187,MR2365643,MR2496438,MR2759182}, etc. In all the literatures mentioned above, the noises considered are white in time. The hitting properties for the case of noise being fractional in time are often more difficult to obtain, due to the complex covariance structure associated to the noise. For the relevant work, we are only aware of \cite{MR2513117,MR4333508} for systems of stochastic heat equations and \cite{MR3231213} for systems of stochastic wave equations. The present work aims to make further contribution in this direction.

It is universally accepted that one of fundamental ingredients to derive the bounds of hitting probabilities for a random field concerns a detailed analysis for the canonical pseudo-distance associated to the random field. For our setting, we prove that for any $(t,x),(s,y)\in[t_0,T]\times[-M,M]^d$, the associated canonical pseudo-distance $\|u(t, x)-u(s, y)\|_{L^2(\Omega)}$ compares with
\begin{align} \label{joint_reg}
\varrho\big((t,x),(s,y)\big)=|t-s|^{\alpha_2}+\left(\ln \frac{2e\sqrt{d}M}{|x-y|}\right)^{\frac{1}{2}\mathbb 1_{\{\alpha_1=1\}}}|x-y|^{\alpha_1\wedge 1},
\end{align}
where $\alpha_2:=\frac{\alpha_1}{\alpha+\gamma}$. The main obstacle of obtaining this ingredient lies in dealing with the interaction of the fractional noise and the operator $(I-\Delta)^{\alpha/2}(-\Delta)^{\gamma/2}$. In our proofs, several complex calculations are involved in order to handle the singularity of the correlation function associated to the fractional noise and the nonhomogeneous term in the Fourier transform of the Green function which is resulted from the operator $(I-\Delta)^{\alpha/2}$. We note that, unlike most cases, this canonical pseudo-distance has anisotropies described by the function $\varrho$ (which includes a logarithm factor) other than power functions. For this kind of anisotropic Gaussian random fields, some powerful criteria are provided in \cite{MR4333508} to deal with its hitting probabilities, which have been successfully applied to systems of linear stochastic heat equations with fractional noise (see \cite[Section 4]{MR4333508}) and linear stochastic biharmonic heat equations with space-time white noise (see \cite{Hinojosa}). We also adopt the criteria to investigate hitting probabilities of the system of generalized fractional kinetic equations with time-fractional noise. In detail, some second order properties of $u$ are further investigated, which together with the equivalent expression \eqref{joint_reg} of the canonical pseudo-distance yield the bounds of hitting probabilities of $\hat u$ in terms of the $\mathfrak{g}_q$-capacity and $g_q$-Hausdorff measure (see Theorem \ref{hittingprobability}).

As a consequence of Theorem \ref{hittingprobability}, points are non-polar for $\hat u$ if $n<Q:=\frac{1}{\alpha_2}+\frac{d}{\alpha_1\wedge 1}$ and are polar for $\hat u$ if $n>Q$ (see Corollary \ref{polaritypoints}). Therefore, the critical value $Q$ is called the critical dimension for hitting points of $\hat u$ when $Q$ is an integer. However, in the critical case $n = Q$, the issue of polarity for $z\in \mathbb R^{Q}$ is not answered by Theorem \ref{hittingprobability}. Therefore, another purpose of this paper is to discuss the polarity of points for $\hat u$ in the critical dimension $Q$. 
As far as polarity of points in the critical dimension is concerned, it is worthwhile to refer to the seminal work \cite{MR3737922,MR4317713} which resolved the issue of polarity of points in the critical dimension for Gaussian and non-Gaussian random fields, respectively. 
Taking advantage of a harmonizable representation of $u$ and the criteria established by \cite{MR3737922}, we prove that if $\alpha_1\in(0,1)$, points are polar for $\hat{u}$ in the critical dimension $Q$ (see Theorem \ref{critical}). 
This result particularly provides strong evidence for the conjecture on the issue of polarity for singletons in the critical dimension for a linear stochastic biharmonic heat equation raised in \cite{Hinojosa} (see Remark \ref{conjection}).

Finally, we outline the rest of the paper. Section \ref{section2} briefly set up a further description of the structure of the Gaussian noise that we are dealing with and introduce some preliminary material of the potential theory. We also prove the existence and uniqueness of a random field solution $u$ to \eqref{FKE} in Subsection \ref{solution}. Section \ref{holdercon} is devoted to the mean square modulus of continuity of the solution $u$ both in time and in space. Finally, Section \ref{section4} focuses on hitting properties of $\hat u$, including the lower and upper bounds of hitting probabilities and polarity of points in the critical dimension for $\hat u$. 

\vspace{2em}
\noindent{\bf Notations.} In the remainder of the article, as usual, we use $C$ to denote various positive real constants and $C_{a_1,\dots,a_k}$ to emphasize the dependence on certain parameters $a_1,\dots,a_k$. Their values may differ from occurrence to occurrence. Specific constants will be denoted by $C_i,C^{\prime}$ and $c$, etc. The Euclidean norm and inner product on $\br^k$ are denoted by $|\cdot|$ and $\langle \cdot,\cdot\rangle$, respectively. For any integrable function $f:\br^k \to\br$, its Fourier transform is defined by $\mathcal{F} f(\xi)=\int_{\br^k} e^{-\ri \langle\xi, x\rangle} f(x) \ud x.$ We write $\mathscr B(S)$ as
 the Borel $\sigma$-algebra of  a topological space $S$.
As usual, for a complete probability space $(\Omega,\mathscr F,\bp)$, $\ee$ denotes the expectation w.r.t. $\bp$ and $L^p(\Omega;\br^k)$ denotes the space of all $\br^k$-valued random variables with finite $p$-th  moment, endowed with the norm $\|\cdot\|_{p}:=[\ee|\cdot|^p]^{\frac{1}{p}}$. For any $a,b \in \br$, let $a\vee b := \max\{a,b\}$ and $a\wedge b := \min\{a,b\}$. For any set $A$, denote by $\mathbb{1}_{A}$ the indicator function of $A$. For a monotonic increasing function $q$, we denote by $\check{q}$ the inverse function of $q$. Throughout the paper, we always assume $t_{0} \in(0, T)$ and $M>1$ are fixed.

\section{Prerequisites} \label{section2}
This section is devoted to a brief introduction to some basic background knowledge that will be used throughout this paper. We consider first the well-posedness of equation \eqref{FKE} and later a further description of the noise $\dot W^{H,\beta}$. We also provide some basic elements of the potential theory.

\subsection{Well posedness of equation \eqref{FKE}}\label{solution}
Let $d \ge 1$ be a spatial dimension parameter. 
In this paper, the noise $\dot{W}^{H,\beta}$ is the formal partial derivative of a centered Gaussian random field $W^{H,\beta}$ which is fractional in time with a Hurst index $H\in(\frac{1}{2},1)$ and colored ($\beta\in(0,d)$) or white ($\beta=d$) in space with a scaling parameter $\beta$. This type of noise corresponds to the covariance function
\begin{align*}
\ee\left[\dot{W}^{H,\beta}(t, x) \dot{W}^{H,\beta}(s, y)\right] &=\alpha_{H}|t-s|^{2H-2}f(x-y),  \quad (t, x),(s, y) \in[0, T] \times \br^{d},
\end{align*}
where $f$ is a nonnegative and nonnegative definite tempered measure and $\alpha_{H}=H(2H-1)$. According to the Bochner theorem, $f$ is the Fourier transform of a tempered nonnegative measure $\mu$ on $\mathbb{R}^{d}$, which is the so-called spatial spectral measure. 

Next, we shall give a brief introduction to the stochastic integral w.r.t. $W^{H,\beta}$ for deterministic processes. Let $\mathscr{B}_b(\br^d)$ be the collection of Borel measurable subsets of $\br^d$ with finite Lebesgue measure and $\mathcal{E}$ be the set of linear combinations of elementary functions $\mathbb1_{[0, t] \times A}, t \ge 0, A \in \mathscr{B}_b(\br^d)$. Let $\mathcal{H}$ be the completion of $\mathcal E$ w.r.t. the inner product
\begin{align*}
\langle\phi, \psi\rangle_{\mathcal{H}}=&~\alpha_{H}\int_0^T\int_0^T|r-w|^{2H-2} \ud r\ud w\int_{\mathbb{R}^{d}} \int_{\br^{d}} \phi(r,x) f(x-y) \psi(w,y) \ud x \ud y \quad \forall\, \phi, \psi \in \mathcal E.
\end{align*} 
The Plancherel theorem shows that the inner product of $\mathcal H$ can be rewritten as
\begin{align*}
\langle\phi, \psi\rangle_{\mathcal{H}}=&~\alpha_{H}\int_0^T\int_0^T|r-w|^{2H-2} \ud r\ud w\int_{\br^d} \mathcal{F}\phi(r, \cdot)(\xi) \overline{\mathcal{F}\psi(w, \cdot)(\xi)}\mu(\ud\xi)\\
=&~\beta_{H}\int_{\br}|\tau|^{1-2H} \ud \tau\int_{\br^d} \mathcal{F}(\phi\mathbb 1_{[0,T]})(\tau,\xi) \overline{\mathcal{F}(\psi\mathbb 1_{[0,T]})(\tau,\xi)}\mu(\ud\xi) \quad \forall\, \phi, \psi \in \mathcal E
\end{align*}
with $\beta_{H}=\frac{\alpha_{H}\Gamma(H-1/2)}{4^{1-H}\sqrt\pi\Gamma(1-H)}$ and $\Gamma(a)=\int_0^\infty t^{a-1}e^{-t}\ud t~ (a>0)$ denoting the Gamma function.
For the sake of simplicity, we sometimes omit the coefficients $\alpha_{H}$ and $\beta_{H}$ in the expression of the covariance function. Further assumptions on $\mu$ will be specified later.

Consider an $L^{2}(\Omega,\br)$-valued centered Gaussian process $\{W^{H,\beta}(h)\}_{ h\in \mathcal E}$ with covariance
\begin{align}\label{isometry}
\ee[W^{H,\beta}(h_1) W^{H,\beta}(h_2)]=\langle h_1, h_2\rangle_{\mathcal{H}}\quad \forall \, h_1,h_2\in\mathcal E.
\end{align}
The map $h \mapsto W^{H,\beta}(h)$ is an isometry between $\mathcal{E}$ and the Gaussian space of $W^{H,\beta}$, which can be extended to $\mathcal{H}$. This extension defines an isonormal Gaussian process $W^{H,\beta}=\{W^{H,\beta}(h)\}_{ h \in \mathcal{H}} .$ We write
$$
\int_{0}^{T} \int_{\mathbb{R}^{d}} h(t, x) W^{H,\beta}(\ud t, \ud x):=W^{H,\beta}(h) \quad \forall \, h \in \mathcal{H},
$$
which defines the stochastic integral of an element $h \in \mathcal{H}$ w.r.t.  $W^{H,\beta}$.

Now we proceed to establish the well-posedness of  \eqref{FKE}. The random field approach to stochastic partial differential equations initiated in \cite{MR876085} and extended by \cite{MR1684157} has been an active line of research in the past few decades. By virtue of the ideas given by \cite{MR1684157}, we will understand a solution of \eqref{FKE}  to be a jointly measurable adapted process satisfying the integral form
\begin{align}\label{mildsol}
u(t,x)=\int_{\br^d}G(t,x-y)u_0(y)\ud y+\int_0^t\int_{\br^d}G(t-s,x-y)W^{H,\beta}(\ud s, \ud y)
\end{align}
with $G$ being the Green function (called also the fundamental solution) of
\begin{align*}
\partial_{t} G(t, x)+(I-\Delta)^{\alpha/2}(-\Delta)^{\gamma/2} G(t, x)=0.
\end{align*}
Moreover, the Green function $G$ (see e.g., \cite{MR2571321}) can be written as
\begin{align}\label{green}
G(t,x)=\frac{1}{(2\pi)^{d}}\int_{\mathbb{R}^{d}} e^{\ri\langle x, \xi\rangle} \exp \left\{-t |\xi|^{\gamma}\left(1+|\xi|^{2}\right)^{\frac{\alpha}{2}}\right\} \ud \xi
\end{align}
and its Fourier transform $\mathcal{F} G(t,\cdot)(\xi)$ w.r.t. the space variable is given by
\begin{align}\label{fouriertrans}
\mathcal{F} G(t,\cdot)(\xi)=\exp \left\{-t |\xi|^{\gamma}\left(1+|\xi|^{2}\right)^{\frac{\alpha}{2}}\right\}.
\end{align}
For the sake of convenience, denote $$\Psi(\xi):=|\xi|^{\gamma}(1+|\xi|^2)^{\alpha/2}, \quad\xi\in\br^d.$$

To obtain the well-posedness of \eqref{FKE}, we shall make a standard assumption on the spectral measure $\mu$.
\begin{assumption}\label{dalang-condition}
The spectral measure $\mu$ satisfies the following integrability condition
\begin{align}\label{muhyp}
\int_{\br^d}\left(\frac{1}{1+|\xi|^2}\right)^{(\alpha+\gamma)H}\mu(\ud \xi)<\infty.
\end{align}
\end{assumption}
As the following theorem shows, the random field (or mild) solution of \eqref{FKE} in the sense given by \eqref{mildsol} uniquely exists under Assumption \ref{dalang-condition}. And Assumption \ref{dalang-condition} is sharp, in the sense that it is also necessary in the case of additive noise. To simplify the notation, we write 
$$
G\ast u_{0}(t,x):=\int_{\br^d}G(t,x-y)u_0(y)\ud y
$$
and
$$
\bar{u}(t,x):=\int_0^t\int_{\br^d}G(t-s,x-y)W^{H,\beta}(\ud s, \ud y).
$$

\begin{theorem}[Well-posedness] \label{wellpose}
For any $u_0\in L^1({\br^d})$, the process $\{u(t, x)\}_{(t, x) \in[0, T] \times \br^d}$ given by \eqref{mildsol} exists if and only if Assumption \ref{dalang-condition} holds. In this case, for all $p> 0$ and $T>0$,
$$
\sup _{(t,x ) \in[0, T]\times \mathbb{R}^{d}} \ee\left|\bar{u}(t, x)\right|^p<+\infty.
$$
\end{theorem}
\begin{proof}
It follows from \eqref{green} that for all $t\in(0,T]$, $G(t,\cdot)\in C_0(\br^d)$ of all continuous functions vanishing at infinity. For any $u_0\in L^1({\br^d})$, we have
\begin{align*}
\left|\int_{\br^d}G(t,x-y)u_0(y)\ud y\right|\le ||u_0||_{L^1(\br^d)}||G(t,\cdot)||_{L^{\infty}(\br^d)}<\infty,
\end{align*}
which implies that $G\ast u_{0}(t,x)$ is well defined. What remains to be proved is that $G(t-\cdot, x-\cdot) \in \mathcal{H}$ if and only if Assumption \ref{dalang-condition} holds. Define
\begin{align}\label{Jt}
J_t:=&~\alpha_{H} \int_{0}^{t} \int_{0}^{t}\ud r \ud w|r-w|^{2 H-2} \int_{\br^{d}}\mu(\ud \xi)  \mathcal{F} G(t-r, x-\cdot)(\xi)\overline{\mathcal{F} G(t-w, x-\cdot)(\xi)}\notag\\
=&~\alpha_{H}\int_{0}^{t} \int_{0}^{t}\ud r \ud w|r-w|^{2 H-2} \int_{\br^{d}}\mu(\ud \xi)   e^{-(r+w)\Psi(\xi)}.
\end{align}
By virtue of Lemma \ref{estofh}, we obtain
$$C\int_{\br^d}\left(\frac{1}{1+|\xi|^2}\right)^{(\alpha+\gamma)H}\mu(\ud \xi)\le J_t\le C^\prime\int_{\br^d}\left(\frac{1}{1+|\xi|^2}\right)^{(\alpha+\gamma)H}\mu(\ud \xi)$$
for some $C,C^\prime>0$, which implies the sufficiency and necessity of Assumption \ref{dalang-condition}.

For arbitrary $p>0$, it follows from the isometry formula \eqref{isometry} and the fact that $\bar{u}(t,x)$ is Gaussian that 
$$
\sup _{t \in[0, T], x \in \mathbb{R}^{d}}\ee|\bar{u}(t, x)|^{p}= \sup _{t \in[0, T], x \in \mathbb{R}^{d}}C_p\|\bar{u}(t, x)\|_{2}^p=C_p\sup_{t \in[0, T]} J_{t}^{p/2}=C_pJ_{T}^{p/2}<\infty
$$
with $C_p:=\frac{2^{p/2}}{\pi^{1/2}}\Gamma(\frac{p+1}{2})$.
\end{proof}

The following estimates are crucial to prove the well-posedness and the spatial regularity of \eqref{FKE}. Their proofs are similar to that of \cite[Proposition 4.3]{MR2728174}, and we omit them. 
\begin{lemma}\label{estofh}
Let 
\begin{align*}
N_t^{*}(\xi):=\alpha_{H}\int_0^t\int_0^t|r-w|^{2H-2}e^{-(r+w)|\xi|^{\gamma+\alpha}}\ud r\ud w
\end{align*}
and
$$
N_t(\xi):=\alpha_{H}\int_0^t\int_0^t|r-w|^{2H-2}e^{-(r+w)|\xi|^{\gamma}(1+|\xi|^2)^{\alpha/2}}\ud r\ud w.
$$
Then for any $t\ge0,\xi\in\br^d$,
\begin{align}\label{Nt*bound}
\frac{1}{4}(t\wedge 2^{-1})^{2H}\left(\frac{1}{1+|\xi|^2}\right)^{(\alpha+\gamma)H}\le N_t^{*}(\xi)\le C_{H}(t^{2H}+1)\left(\frac{1}{1+|\xi|^2}\right)^{(\alpha+\gamma)H}
\end{align}
and
\begin{align}\label{Ntbound}
\frac{1}{4}(t\wedge 2^{-1})^{2H}\left(\frac{1}{1+|\xi|^2}\right)^{(\alpha+\gamma)H}\le N_t(\xi)\le C_{H}(t^{2H}+1)\left(\frac{1}{1+|\xi|^2}\right)^{(\alpha+\gamma)H}.
\end{align}
\end{lemma}

\subsection{Time-fractional noise}
In order to cover more types of noise, we will work under the following assumptions exclusively.
\begin{assumption} \label{noise}
The spectral measure $\mu$ satisfies the following conditions: 
\begin{itemize}
\item[(i)] $\mu$ is absolutely continuous w.r.t. the Lebesgue measure, i.e., $\mu(\ud\xi)=\Upsilon(\xi)\ud \xi$.
\item[(ii)] $\Upsilon$ satisfies the following scaling property for some $\beta \in(0, d]$:
\begin{align}\label{scaling}
\Upsilon(c \xi)=c^{\beta-d} \Upsilon(\xi) \quad \text { for all } c>0, \xi \in \br^{d}.
\end{align}
\item[(iii)] The following limits holds for some $C>0$,
\begin{align}
\lim_{|\xi|\to 0}\Upsilon(\xi)=\begin{cases}
\infty, \quad &\beta\in(0,d), \\
C, \quad &\beta=d.
\end{cases} 
\end{align}
\end{itemize}
\end{assumption}
\begin{remark} 
Under Assumption \ref{noise}, the scaling property of $\Upsilon$ implies that $\mu(c A)=c^{\beta} \mu(A)$
for all $c>0$ and  $A\in \mathscr{B}\left(\br^{d}\right)$. It follows from the scaling property that the value of $\Upsilon$ at one point completely determines the values of all points on the ray emanating from the origin in which it is located. In other words, $\Upsilon$ is completely determined by its restriction to the unit sphere $\mathbb S^{d-1}$. 
\end{remark}
In what follows, we will make frequent use of the following properties of $\Upsilon$.
\begin{lemma} \label{Upsilon}
Under Assumption \ref{noise}, $\Upsilon$ satisfies the following properties:
\begin{itemize}
\item[(i)] $\Upsilon(\xi)>0$ for any $\xi\in\br^d\setminus\{0\}$.
\item[(ii)] Its restriction to the unit sphere $\mathbb S^{d-1}$ has a positive lower bound, i.e., there exists a constant $C>0$, such that $\Upsilon(\xi)>C$ for any $\xi\in \mathbb S^{d-1}$.
\item[(iii)] $\Upsilon$ is integrable on $\mathbb S^{d-1}$, i.e., $\int_{\mathbb S^{d-1}}\Upsilon(\xi)\sigma(\ud \xi)<\infty$ with $\sigma(\ud \xi)$ being the surface area measure on $\mathbb S^{d-1}$.  
\end{itemize}
\end{lemma}
\begin{proof}
(i) It is sufficient to prove that $\Upsilon(\xi)>0$ for any $\xi\in\mathbb S^{d-1}$.  Assume by contradiction that there is a $\xi_0\in \mathbb S^{d-1}$ such that $\Upsilon(\xi_0)=0$. Then we have $\Upsilon(c\xi_0)=c^{\beta-d}\Upsilon(\xi_0)=0$ for any $c>0$. Letting $c\to0$, it then contradicts to Assumption \ref{noise} (iii).

(ii) Assume by contradiction that for any $n\in \mathbb N_+$, there is a $\xi_n\in \mathbb S^{d-1}$ such that $\Upsilon(\xi_n)<\frac{1}{n}$. Taking $a\in(0,\frac{1}{d-\beta})$ for the case of $\beta\in(0,d)$ and arbitary $a>0$ for the case of $\beta=d$, we have 
$$
\Upsilon(\xi_nn^{-a})=n^{a(d-\beta)}\Upsilon(\xi_n)<n^{a(d-\beta)-1},
$$
which implies that $\lim_{n\to \infty}\Upsilon(\xi_nn^{-a})=0$.  It then gives rise to a contradiction since $|\xi_nn^{-a}|\to 0$ as $n\to\infty$.

(iii) Since $\mu$ is tempered, $\mu(\{|\xi | \le 1\}) < \infty$. By using polar coordinates and the scaling property, we obtain
\begin{align*}
\mu(\{|\xi | \le 1\})=\int_0^1\int_{\mathbb S^{d-1}}\Upsilon(r\xi)\sigma(\ud \xi)r^{d-1}\ud r=\beta^{-1}\int_{\mathbb S^{d-1}}\Upsilon(\xi)\sigma(\ud \xi),
\end{align*}
which yields that the integral of $\Upsilon$ on the unit sphere is finite.
\end{proof}
\begin{remark}
Under Assumption \ref{noise}, the integrability condition \eqref{muhyp} is equivalent to $\beta\in\big(0,2(\alpha+\gamma)H\big)$ or $\alpha_1>0.$
Indeed, \eqref{muhyp} is equivalent to
\begin{align*}
\int_{|\xi|\le 1}\mu(\ud \xi)<\infty\quad \text{ and }\quad \int_{|\xi|\ge 1}\frac{1}{|\xi|^{2(\alpha+\gamma)H}}\mu(\ud \xi)<\infty.
\end{align*}
The first one holds since $\mu$ is tempered and the second one coincides with
\begin{align*}
\int_{|\xi|\ge 1}\frac{1}{|\xi|^{2(\alpha+\gamma)H}}\mu(\ud \xi)=&~\int_1^{\infty}r^{-2(\alpha+\gamma)H+d-1}\ud r \int_{\mathbb S^{d-1}}\Upsilon(rz)\sigma(\ud z)\\
=&~\int_{\mathbb S^{d-1}}\Upsilon(z)\sigma(\ud z)\int_1^{\infty}r^{-2(\alpha+\gamma)H+\beta-1}\ud r, 
\end{align*}
which is finite if and only if $\beta<2(\alpha+\gamma)H$. The equivalence can be also obtained by decomposing $\{|\xi|>1\}$ into $\{2^k<|\xi|\le 2^{k+1}\}_{k=0}^{\infty}$ and the scaling property of $\mu$.
\end{remark}

Assumption \ref{noise} is satisfied by the following typical  noises, including the white noise, the Riesz kernel noise, fractional noise and a hybrid of these noises, etc.
\begin{example}[White noise]\label{white}
$\Upsilon(\xi)=(2\pi)^{-d}$, i.e., $\mu$ is a multiple of the Lebesgue measure. For this case, $\beta=d$ and the corresponding noise is called spatial white noise.
\end{example}
\begin{example}[Riesz kernel noise]\label{Riesz}
 $\Upsilon(\xi)=C_{d, \beta}|\xi|^{-(d-\beta)} $ for some $ \beta \in(0, d)$, where
\begin{align}\label{constant}
C_{d, \beta}=\pi^{-d / 2} 2^{-\beta} \frac{\Gamma((d-\beta) / 2)}{\Gamma(\beta / 2)} .
\end{align}
 We refer to this kind of noise as a spatial Riesz kernel noise with parameter $\beta$.
\end{example}
\begin{example}[Fractional noise]\label{fractionalnoise}
$\Upsilon(\xi)=\prod_{i=1}^{d}\beta_{H_{i}}\left|\xi_ {i}\right|^{1-2H_{i}}$  for some  $H_{1}, \ldots, H_{d} \in\left(\frac{1}{2}, 1\right) $, where $\beta_{H_i}=\frac{\alpha_{H_i}\Gamma(H_i-1/2)}{4^{1-H_i}\sqrt\pi\Gamma(1-H_i)}$. The function  $\Upsilon$  satisfies the scaling relation \eqref{scaling} with $ \beta=2d-2\sum_{i=1}^{d}H_{i} $. This noise corresponds to the fractional Brownian sheet with indices $H_{1}, \ldots, H_{d} $.
\end{example}

One can certainly construct examples by combining the previous three examples.
\begin{example}[Hybrid noise]
One can partition the  $d$  coordinates into  $k \in\{1, \cdots, d\} $ groups with  $d_{i} \geq 1 $ coordinates in the  $i$-th group such that  $d=d_{1}+\cdots+d_{k} $. Amount each group of coordinates, one imposes a Riesz kernel with parameter  $\beta_{i} \in\left(0, d_{i}\right]$, that is,  $\Upsilon(\xi)= \prod_{i=1}^{k}C_{d_i,\beta_i}\left|\xi^{(i)}\right|^{-\left(d_{i}-\beta_{i}\right)}$ where $ \xi^{(i)} $ denotes the set of coordinates of  $\xi$  in the  $i$-th group and $C_{d_i,\beta_i}$ is defined by \eqref{constant}. Then the scaling property  \eqref{scaling} is satisfied with $ \beta= \sum_{i=1}^{k} \beta_{i} \in(0, d] $. In particular, this  type of noise includes Examples \ref{white}--\ref{fractionalnoise}.
\end{example}

\subsection{Background of potential theory} \label{2.3}
In the last part of this section, we recall some notions of Hausdorff measure and capacity that will be used in Section \ref{section4}. Let $g:\br_+\rightarrow \br_{+}$ be a right continuous and monotonic increasing function which on $\left[0, \varepsilon_{0}\right] $ is strictly increasing with some $\varepsilon_{0}>0$ and $g(0)=0$. The $g$-Hausdorff measure (see e.g. \cite[Chapter 2]{MR1692618}) of $A  \in\mathscr B(\br^k)$ is defined by
$$
\mathcal{H}_{g}(A)=\liminf_{\varepsilon \downarrow 0}\left\{\sum_{i=1}^{\infty} g\left(2 r_{i}\right): A \subset \bigcup_{i=1}^{\infty} B\left(x_i,r_i\right), \sup _{i \geq 1} r_{i} \leq \varepsilon\right\},
$$
where $B(x, r)$ denotes the Euclidean open ball of radius  $r>0$ centered at  $x \in \br^k$. In this paper, we mainly use this notion referred to two cases: 
\begin{itemize}
\item[(i)] $g(\tau)=\tau^\lambda $ with $\lambda >0$. In this situation, $\mathcal H_g$ is the classical $\lambda $-dimensional Hausdorff measure which is usually denoted by $\mathcal H_{\lambda }$. 
\item[(ii)]  $g(\tau)=\tau^{\lambda_{1}}\check{q}(\tau)^{-\vartheta }$ with $q(\tau)=\tau^{\lambda_{2}}\left(\ln \frac{C}{\tau}\right)^{\delta}, \lambda_{1}, \lambda_{2}, \vartheta, \delta>0$. Here $C$ is large enough to ensure $\ln \left(\frac{C}{\tau}\right) \geq 1$.
\end{itemize}
To cover the $\lambda$-dimensional Hausdorff measure with $\lambda<0$, if $g(0)=\infty$, we set $\mathcal{H}_{g}(A)=\infty$.
 
A function $\mathfrak{g}: \br^k \longrightarrow \mathbb{R}_{+} \cup\{\infty\}$ is called a symmetric potential kernel if it is continuous on $\br^k \backslash\{0\}$, symmetric, $\mathfrak{g}(z)>0$ for all $z \neq 0$ and $\mathfrak{g}(0)=\infty$. The  $\mathfrak{g}$-capacity (see e.g. \cite[Appendix D]{MR1914748}) of $A  \in\mathscr B(\br^k)$ is defined by
$$
\operatorname{Cap}_{\mathfrak{g}}(A)=\left[\inf _{\nu \in \mathscr{P}(A)} \mathcal{E}_{\mathfrak{g}}(\nu)\right]^{-1},
$$
where $\mathcal{E}_{\mathfrak{g}}(\nu):=\iint_{\br^k \times \br^k} \mathfrak{g}(y-z) \nu(\ud y) \nu(\ud z)$ and  $\mathscr{P}(A)$ denotes the set of probability measures on $A$. As usual, we use the conventions that $\inf\emptyset:=\infty$ and $1/0:=\infty$.
$\operatorname{Cap}_\mathfrak{g}$ is called the capacity relative to a symmetric potential kernel $\mathfrak{g}$. By convention, we set $\operatorname{Cap}_{\mathfrak{g}}(A)=1$ for the case of $\mathfrak{g}(0) \in[0, \infty)$. In this article, we will use this notion with $\mathfrak{g}=1/g(|\cdot|)$, where $g$ is as in cases (i) and (ii) above. Note that, in the case (i), the $\mathfrak{g}$-capacity is the $\lambda$-dimensional Bessel--Riesz capacity, usually denoted by $\mathrm{Cap}_{\lambda}$.

\section{Regularity of the solution}\label{holdercon}
In this section, we are devoted to the regularity estimate of the solution \eqref{mildsol} of \eqref{FKE}. To be specific, we are interested in the behavior of the increments of the solution $u(t,x)$ w.r.t. the temporal variable $t$ and the spatial variable $x$, and give the optimal upper and lower bounds for these increments in mean square sense. The following proposition shows the joint regularity of $G\ast u_{0}$.
\begin{proposition}
For any $u_0\in L^1(\br^d)$, the function $G\ast u_{0}:[t_0,T]\times \br^d\to\br$ is Lipschitz continuous.
\end{proposition}
\begin{proof}
Using the definition of $G$ in \eqref{green}, we see that for any $t_0\le s\le t\le T$ and $x, y\in\br^d$,
\begin{align*} 
|G(t,x)-G(t,y)|=&~\frac{1}{(2\pi)^{d}}\left|\int_{\mathbb{R}^{d}}\left( e^{\ri\langle x, \xi\rangle} -e^{\ri\langle y, \xi\rangle} \right)\exp \left\{-t \Psi(\xi)\right\} \ud \xi\right|\\
\le &~\frac{|x-y|}{(2\pi)^{d}}\int_{\mathbb{R}^{d}} |\xi| \exp \left\{-t_0  \Psi(\xi)\right\} \ud \xi\\
\le&~C_{d,\alpha,\gamma,t_0}|x-y|
\end{align*}
and
\begin{align*} 
|G(t,y)-G(s,y)|=&~\frac{1}{(2\pi)^{d}}\left|\int_{\mathbb{R}^{d}}e^{\ri\langle y, \xi\rangle}\left(  e^{-t  \Psi(\xi)}-e^{-s \Psi(\xi)} \right)\ud \xi\right|\\
\le&~\frac{|t-s|}{(2\pi)^{d}}\int_{\mathbb{R}^{d}}  \Psi(\xi)e^{-t_0 \Psi(\xi)}\ud \xi \\
\le&~C_{d,\alpha,\gamma,t_0}|t-s|.
\end{align*}
Those together with the definition of $G\ast u_{0}$ and $u_0\in L^1(\br^d)$ yield the Lipschitz continuity of $G\ast u_{0}$.
\end{proof}

Recall that $\alpha_1=(\alpha+\gamma)H-\beta/2$. The next proposition analyzes the spatial increments of the random field $\bar{u}$ in mean square sense.
\begin{proposition}\label{srs}
Under Assumptions \ref{dalang-condition} and \ref{noise}, there exist two positive constants $C_{1}$ and $C_{2}$ such that \begin{align}\label{spaceH}
C_{1}q_1(|x-y|) \le \|\bar{u}(t, x)-\bar{u}(t, y)\|_2\le C_{2}q_1(|x-y|)
\end{align}
for all $x, y \in [-M,M]^d$ and $t \in[t_0, T]$, where $q_1:[0,2\sqrt d M]\to\br_+$ is defined as
$$q_1(\tau):=\left(\ln \frac{C_{d,M}}{\tau}\right)^{\frac{1}{2}\mathbb 1_{\{\alpha_1=1\}}}\tau^{\alpha_1 \wedge 1}$$
with $C_{d,M}=2e\sqrt d M$.
\end{proposition}
\begin{proof}
By \eqref{fouriertrans}, we get
\begin{align*}
&~\ee|\bar{u}(t, x)-\bar{u}(t, y)|^{2}=\|G(t-\cdot,x-\cdot)-G(t-\cdot,y-\cdot)\|_{\mathcal H}^2\\
=&~\alpha_{H}\int_{0}^{t} \int_{0}^{t} \ud r\ud w|r-w|^{2H-2}   \int_{\mathbb{R}^{d}} \mu(\ud \xi)e^{-(2t-r-w)\Psi(\xi)}|e^{-\mathrm{i}\langle x, \xi\rangle}-e^{-\mathrm{i}\langle y, \xi\rangle}|^2\\
=& ~2\int_{\mathbb{R}^{d}} \mu(\ud \xi)N_t(\xi)(1-\cos\langle x-y, \xi\rangle).
\end{align*}

\textbf{Lower bound}: The proof of the lower bound of \eqref{spaceH} is divided into two cases.

\textbf{Case (i):} $|x-y|< 1-\epsilon$ with sufficiently small $\epsilon>0$.
Using \eqref{Ntbound} and polar coordinates $\xi=\rho z$ and making the change of variables $r=\rho|x-y|$, we derive
\begin{align}\label{lowerb1}
&~\mathbb{E}|\bar{u}(t, x)-\bar{u}(t, y)|^{2} \notag\\
\ge&~ C_{H, t_0} \int_{|\xi| \geq 1} \mu(\ud \xi)\left(\frac{1}{1+|\xi|^{2}}\right)^{(\alpha+\gamma) H}\sin^2\left(\frac{1}{2} \langle x-y,\xi\rangle\right) \notag\\
\ge&~ C_{H, t_0,\alpha,\gamma} \int_{|\xi| \geq 1} \mu(\ud \xi) |\xi|^{-2(\alpha+\gamma) H}\sin^2\left(\frac{1}{2} \langle x-y,\xi\rangle\right)\notag\\
\ge &~ C_{H, t_0,\alpha,\gamma} |x-y|^{2\alpha_1} \int_{\mathbb{S}^{d-1}} \int_{|x-y|}^{1} r^{-2\alpha_1-1}\sin^2\left(\frac{r}{2}\Big\langle z, \frac{x-y}{|x-y|}\Big\rangle\right) \Upsilon(z)\ud r \sigma(\ud z)\notag\\
\ge&~ C_{H, t_0,\alpha,\gamma}\frac{\sin^2 1}{4} |x-y|^{2\alpha_1}  \int_{|x-y|}^{1} r^{-2\alpha_1+1}\ud r\int_{\mathbb{S}^{d-1}} \Big\langle z, \frac{x-y}{|x-y|}\Big\rangle^2 \sigma(\ud z),
\end{align}
where the last step used Lemma \ref{Upsilon} (ii) and the fact that
$\theta\mapsto\frac{\sin \theta}{\theta}$ is decreasing  on $(0, 1)$.
 For any
  $c_0\in\left(0,\frac{\ln\frac{1}{1-\epsilon}}{\ln\frac{C_{d,M}}{1-\epsilon}}\right)$ and
  $c_1\in(0,1-(1-\epsilon)^{2\alpha_1-2})$, we have
 \begin{align*}
\int_{|x-y|}^1r^{-2\alpha_1+1}\ud r=
\left\{
\begin{array}{lll}
\ln \frac{1}{|x-y|}\ge c_0\ln\frac{C_{d,M}}{|x-y|}\qquad\quad\quad~~~~\,\quad\text{if}\quad\,\,\alpha_1=1,\\
 \frac{|x-y|^{2-2\alpha_1}-1}{2\alpha_1-2}\ge
 \left\{\begin{array}{lll}
 \frac{c_1|x-y|^{2-2\alpha_1}}{2\alpha_1-2} & \text { if } & \alpha_1>1, \\
\frac{1-(1-\epsilon)^{2-2\alpha_1}}{2-2\alpha_1}& \text { if } & \alpha_1\in(0,1).
\end{array}\right.
\end{array}
\right.
\end{align*}
Next, we estimate the integral w.r.t. $z$ in \eqref{lowerb1}. Let $F(z)=\Big\langle z, \frac{x-y}{|x-y|}\Big\rangle\frac{x-y}{|x-y|}$ be an $\br^d$-valued function. Availing oneself of Green’s formula, we obtain 
\begin{align*}
\int_{\mathbb{S}^{d-1}} \Big\langle z, \frac{x-y}{|x-y|}\Big\rangle^{2}\sigma(\ud z)=~\int_{\mathbb{S}^{d-1}} \langle z, F(z)\rangle\sigma(\ud z)=~\int_{|z|\le 1}\nabla\cdot F(z)\ud z=\omega_d
\end{align*}
with $\omega_d=\frac{2\pi^{\frac{d}{2}}}{d\Gamma(\frac{d}{2})}$ being the volume of the unit ball in $\br^d$.
 
Substituting those into \eqref{lowerb1} gives the desired lower bound for $|x-y|<1-\epsilon$.

\textbf{Case (ii): $|x-y|\ge1-\epsilon$.} Let $k=\arg\max_i|x_i-y_i|$.
 Using \eqref{Ntbound} and the fact $\frac{1}{1+\theta^2}\ge\frac{1}{(1+r^2)\theta^2}$ for all $|\theta|\ge\frac{1}{r}$, we derive
\begin{align*}
&~\ee|\bar{u}(t, x)-\bar{u}(t, y)|^{2}\\
\ge &~C_{t_0,H}\int_{|\xi|\ge |x_k-y_k|^{-1}} \mu(\ud \xi)\left(\frac{1}{1+|\xi|^2}\right)^{(\alpha+\gamma)H}(1-\cos\langle x-y, \xi\rangle)\\
\ge&~\frac{C_{t_0,H}}{(1+|x_k-y_k|^2)^{(\alpha+\gamma)H}}\int_{|\xi|\ge |x_k-y_k|^{-1}} \mu(\ud \xi)|\xi|^{-2(\alpha+\gamma)H}(1-\cos\langle x-y, \xi\rangle).
\end{align*}
By the change of variables $\eta=|x_k-y_k|\xi$, we have
\begin{align*}
&~\ee|\bar{u}(t, x)-\bar{u}(t, y)|^{2}\\
\ge &~\frac{C_{t_0,H}|x_k-y_k|^{2\alpha_1}}{(1+|x_k-y_k|^2)^{(\alpha+\gamma)H}}\int_{|\eta|\ge 1}\mu(\ud \eta)\frac{1-\cos\big(\sum_{j\ne k}\frac{x_j-y_j}{|x_k-y_k|}\eta_j+\mathrm{sgn}(x_k-y_k)\eta_k\big)}{|\eta|^{2(\alpha+\gamma)H}}\\
\ge&~ C_{t_0,H,M}|x-y|^{2\alpha_1},
\end{align*}
where in the last line we used $|x-y|\le \sqrt d|x_k-y_k|$ and the fact 
$$
\inf_{\max_{j\ne k}|\theta_j|\le1}\int_{|\eta|\ge 1}\frac{1-\cos(\sum_{j\ne k}\theta_j\eta_j\pm \eta_k)}{|\eta|^{2(\alpha+\gamma)H}}\mu(\ud \eta)>0,
$$
since the infimum is taken over a compact set for a positive and continuous function. For $|x-y|\ge1-\epsilon$, we have $|x-y|^{2\alpha_1}\ge |x-y|^{2}(1-\epsilon)^{2(\alpha_1-1)}$ for $\alpha_1>1$, and  $|x-y|^{2\alpha_1}\ge \left(\ln \frac{C_{d,M}}{1-\epsilon}\right)^{-1}|x-y|^2\ln \frac{C_{d,M}}{|x-y|}$ for $\alpha_1=1$, which implies the desired lower bound for $|x-y|\ge 1-\epsilon$.

\textbf{Upper bound}: To show the upper bound, using \eqref{Ntbound} again, we obtain
\begin{align*}
\ee|\bar{u}(t, x)-\bar{u}(t, y)|^{2}
\le C_{H}(t^{2H}+1)\int_{\br^d} \mu(\ud \xi)\left(\frac{1}{1+|\xi|^2}\right)^{(\alpha+\gamma)H}(1-\cos\langle x-y, \xi\rangle).
\end{align*}
Considering the polar coordinate transformation $\xi=\rho z$ and using the change of variables $r=\rho|x-y|$, we obtain
\begin{align}\label{upperbound}
&~\ee|\bar{u}(t, x)-\bar{u}(t, y)|^{2}\notag\\
\le &~C_{H, T}|x-y|^{2\alpha_1} \int_{\mathbb{S}^{d-1}} \int_{0}^{\infty} r^{\beta-1}\left(\frac{1}{|x-y|^{2}+r^{2}}\right)^{(\alpha+\gamma) H}\Upsilon(z)\notag\\ 
&~\qquad\qquad\qquad\qquad\qquad\qquad\times\left(1-\cos \left(r\Big\langle z, \frac{x-y}{|x-y|}\Big\rangle\right)\right) \ud r \sigma(\ud z).
\end{align}
We divide the estimate of the integral w.r.t. $r$ in \eqref{upperbound}  into two cases as follows.

{\bf Case 1:} $\alpha_1> 1$. By the inequality $1-\cos \theta\le \frac{\theta^2}{2}$ for $\theta\in\br$, we obtain that for $z\in \mathbb S^{d-1}$,
\begin{align*}
 &~\int_{0}^{\infty} r^{\beta-1}\left(\frac{1}{|x-y|^{2}+r^{2}}\right)^{(\alpha+\gamma) H} \left(1-\cos \left(r\Big\langle z, \frac{x-y}{|x-y|}\Big\rangle\right)\right) \ud r\\
  \le&~\frac{1}{2}\int_{0}^{\infty} r^{\beta+1}\left(\frac{1}{|x-y|^{2}+r^{2}}\right)^{(\alpha+\gamma) H}\ud r\\
 \le&~\frac{1}{2}|x-y|^{-2\alpha_1}\int_{0}^{|x-y|} r \ud r+\frac{1}{2}\int_{|x-y|}^{\infty} r^{\beta+1}\left(\frac{1}{r^{2}}\right)^{(\alpha+\gamma) H} \ud r\\
 \le&~C_{H,\alpha,\gamma}|x-y|^{2(1-\alpha_1)}.
\end{align*}

{\bf Case 2:} $\alpha_1\in(0, 1]$. 
By the inequality $1-\cos \theta\le 2\wedge \frac{\theta^2}{2}$ for $\theta\in\br$, we obtain 
\begin{align*}
 &~\int_{0}^{\infty} r^{\beta-1}\left(\frac{1}{|x-y|^{2}+r^{2}}\right)^{(\alpha+\gamma) H} \left(1-\cos \left(r\Big\langle z, \frac{x-y}{|x-y|}\Big\rangle\right)\right) \ud r\\
   \le&~\frac{1}{2}\int_{0}^{1} r^{\beta+1}\left(\frac{1}{|x-y|^{2}+r^{2}}\right)^{(\alpha+\gamma) H}\ud r+2\int_{1}^{\infty} r^{\beta-1-2(\alpha+\gamma) H} \ud r\\
    \le&~\frac{1}{2}\int_{0}^{1} r^{\beta+1}\left(\frac{1}{|x-y|^{2}+r^{2}}\right)^{(\alpha+\gamma) H}\ud r+\alpha_1^{-1}
\end{align*} 
for any $z\in \mathbb S^{d-1}$. Notice that for $\alpha_1=1$,
  \begin{align*}
  &~\int_{0}^{1} r^{\beta+1}\left(\frac{1}{|x-y|^{2}+r^{2}}\right)^{(\alpha+\gamma) H}\ud r\\
  \le&~\int_{0}^{|x-y|} r^{\beta+1}|x-y|^{-2(\alpha+\gamma) H}\ud r+\int_{|x-y|}^{1} r^{\beta+1-2(\alpha+\gamma) H}\ud r\\
  \le&~|x-y|^{-2}\int_{0}^{|x-y|} r \ud r+\int_{|x-y|}^1 r^{-1}\ud r\le \frac{1}{2}+\ln\frac{1}{|x-y|},\quad\text{if}\quad|x-y|< 1,
\end{align*}  
and 
  \begin{align*}
  &~\int_{0}^{1} r^{\beta+1}\left(\frac{1}{|x-y|^{2}+r^{2}}\right)^{(\alpha+\gamma) H}\ud r\\
 \le&~ \int_{0}^{1} r^{\beta+1}|x-y|^{-2(\alpha+\gamma) H}\ud r\le \frac{1}{2},\quad\text{if}\quad|x-y|\ge 1.
\end{align*}
While for $\alpha_1\in(0,1)$,
\begin{align*}
&\int_{0}^{1} r^{\beta+1}\left(\frac{1}{|x-y|^{2}+r^{2}}\right)^{(\alpha+\gamma) H}\ud r\le \int_{0}^{1} r^{\beta+1-2(\alpha+\gamma) H} \ud r<\infty.
\end{align*}

Substituting the estimates in Case 1 and Case 2 into \eqref{upperbound} and using Lemma \ref{Upsilon} (iii),  then the upper bound of the spatial increments has been proved.
\end{proof}

The next proposition shows that the optimal mean square H\"{o}lder exponent in temporal direction of $\bar u$ is $\alpha_2:=\frac{\alpha_1}{\alpha+\gamma}$.. 
\begin{proposition}\label{srt}
Under Assumptions \ref{dalang-condition} and \ref{noise},  there exist two positive constants $C_{3}$ and $C_{4}$ such that 
\begin{align}\label{timeH}
C_3|s-t|^{\alpha_2}\le\|\bar{u}(t, x)-\bar{u}(s, x)\|_2 \le C_4|s-t|^{\alpha_2}\quad \forall ~t, s \in\left[t_{0}, T\right], x \in \br^{d}.
\end{align}
Moreover, for all $t, s \in[t_0, T], x,y \in \br^{d}$,
\begin{align}\label{timeHS}
\|\bar{u}(t, x)-\bar{u}(s, y)\|_2\ge C_3|s-t|^{\alpha_2}.
\end{align}
\end{proposition}
\begin{proof}
Without loss of generality, let $t_0\le s< t\le T$. Then we have
\begin{align*}
&~\ee|\bar{u}(t, x)-\bar{u}(s, x)|^{2} \\
=\;&~ \alpha_{H}\int_{0}^{t} \ud r\int_{0}^{t}  \ud w|r-w|^{2 H-2} \int_{\mathbb{R}^{d}} \mu(\ud \xi) e^{-(t-r)\Psi(\xi)} e^{-(t-w)\Psi(\xi)} \\ 
&-2 \alpha_{H}\int_{0}^{t} \ud r\int_{0}^{s}  \ud w|r-w|^{2 H-2} \int_{\br^d} \mu(\ud \xi) e^{-(t-r)\Psi(\xi)} e^{-(s-w)\Psi(\xi)} \\ 
&+\alpha_{H} \int_{0}^{s}\ud r \int_{0}^{s}  \ud w|r-w|^{2 H-2} \int_{\mathbb{R}^{d}} \mu(\ud \xi) e^{-(s-r)\Psi(\xi)} e^{-(s-w)\Psi(\xi)} \\
=\;&~ \alpha_{H} \int_{s}^{t} \ud r \int_{s}^{t} \ud w|r-w|^{2 H-2} \int_{\br^{d}} \mu(\ud \xi) e^{-(t-r)\Psi(\xi)} e^{-(t-w)\Psi(\xi)} \\ 
&+\alpha_{H} \int_{0}^{s} \ud r \int_{0}^{s} \ud w|r-w|^{2 H-2} \int_{\mathbb{R}^{d}} \mu(\ud \xi) e^{(r+w)\Psi(\xi)}\big(e^{-t\Psi(\xi)}-e^{-s\Psi(\xi)}\big)^2\\ 
&+2 \alpha_{H} \int_{s}^{t} \ud r \int_{0}^{s} \ud w|r-w|^{2 H-2} \int_{\br^d} \mu(\ud \xi)e^{-(t-r)\Psi(\xi)} \big(e^{-(t-w)\Psi(\xi)}-e^{-(s-w)\Psi(\xi)}\big) \\
=:&~ R_1(t, s)+R_2(t, s)+R_3(t, s) . 
\end{align*}

\textbf{Upper bound}:
Note that for $t\ge s$, $R_3(t, s)\le 0.$
Using the change of variables $(r_1,w_1)=(\frac{t-r}{t-s},\frac{t-w}{t-s})$ and $\eta=(t-s)^{\frac{1}{\gamma+\alpha}}\xi$ in turn, we obtain
\begin{align*}
R_1(t, s) =& ~\alpha_{H} |t-s|^{2H}\int_{0}^{1} \ud r_1 \int_{0}^{1} \ud w_1|r_1-w_1|^{2 H-2} \int_{\br^{d}} \mu(\ud \xi) e^{-(r_1+w_1)(t-s)\Psi(\xi)} \\ 
\le&~ \alpha_{H} |t-s|^{2H}\int_{0}^{1} \ud r_1 \int_{0}^{1} \ud w_1|r_1-w_1|^{2 H-2} \int_{\br^{d}} \mu(\ud \xi) e^{-(r_1+w_1)(t-s)|\xi|^{\gamma+\alpha}}\\ 
 =&~  |t-s|^{2\alpha_2}\int_{\mathbb{R}^{d}} N_1^*(\eta)\mu(\ud \eta) .
\end{align*}
In view of \eqref{Nt*bound} and \eqref{muhyp}, we derive
\begin{align*}
R_1(t, s)\le ~C_{H,\gamma,\alpha,d,T} |t-s|^{2\alpha_2}.
\end{align*}

By means of the change of variables $(r_1,w_1)=\big((s-r)\Psi(\xi),(s-w)\Psi(\xi)\big)$, we obtain
\begin{align*}
&~R_2(t, s)\\
= & ~\alpha_{H} \int_{0}^{s\Psi(\xi)} \ud r_1 \int_{0}^{s\Psi(\xi)} \ud w_1 |r_1-w_1|^{2 H-2}e^{-(r_1+w_1)} \int_{\mathbb{R}^{d}} \mu(\ud \xi)\Psi(\xi)^{-2H}\big(e^{-(t-s)\Psi(\xi)}-1\big)^2\\
\le & ~\alpha_{H} \int_{0}^{\infty} \ud r_1 \int_{0}^{\infty} \ud w_1 |r_1-w_1|^{2 H-2}e^{-(r_1+w_1)} \int_{\mathbb{R}^{d}} \mu(\ud \xi)\Psi(\xi)^{-2H}\big(e^{-(t-s)\Psi(\xi)}-1\big)^2.
\end{align*}
Note that for any $x>0$, if $(1+x^2)^{\frac{\alpha+\gamma}{2}}>x^{\gamma}(1+x^2)^{\frac{\alpha}{2}}>a>0$, then $x>\sqrt{a^{\frac{2}{\alpha+\gamma}}-1}>\frac{1}{2}a^{\frac{1}{\alpha+\gamma}}$ provided that $a>(\frac{4}{3})^{\frac{\alpha+\gamma}{2}}$. And if
$x^{\alpha+\gamma}<x^{\gamma}(1+x^2)^{\frac{\alpha}{2}}<a$, then $x<a^{\frac{1}{\alpha+\gamma}}$. Therefore, for the case of $|t-s|<(\frac{3}{4})^{\frac{\alpha+\gamma}{2}}$, by the inequality $1-e^{-\theta}\le 1\wedge \theta^{H}$ for $\theta>0$ and the scaling property \eqref{scaling}, 
\begin{align*}
&~\int_{\mathbb{R}^{d}} \Psi(\xi)^{-2H}\big(e^{-(t-s)\Psi(\xi)}-1\big)^2\mu(\ud \xi)\\
\le &~\int_{\Psi(\xi)\le |t-s|^{-1}} \Psi(\xi)^{-2H}|t-s|^{2H}\Psi(\xi)^{2H}\mu(\ud \xi)+\int_{\Psi(\xi)\ge |t-s|^{-1} }\Psi(\xi)^{-2H}\mu(\ud \xi)\\
\le &~|t-s|^{2H}\int_{|\xi|\le |t-s|^{-\frac{1}{\alpha+\gamma}}}\mu(\ud \xi)+\int_{|\xi|\ge \frac{1}{2}|t-s|^{-\frac{1}{\alpha+\gamma}}}|\xi|^{-2H(\alpha+\gamma)}\mu(\ud \xi)\\
=&~|t-s|^{2\alpha_2}\int_{|\eta|\le 1}\mu(\ud \eta)+|t-s|^{2\alpha_2}\int_{|\eta|\ge \frac{1}{2}}|\eta|^{-2H(\alpha+\gamma)}\mu(\ud \eta)\\
\le&~C|t-s|^{2\alpha_2}.
\end{align*}
This and 
$
 \int_{0}^{\infty} \ud r \int_{0}^{\infty} \ud w~ |r-w|^{2 H-2}e^{-r}e^{-w} =\Gamma(2H-1)
$
allow us to conclude that $R_2(t, s)\le C|t-s|^{2\alpha_2}$, which completes the proof of the upper bound for the case of $|t-s|<(\frac{3}{4})^{\frac{\alpha+\gamma}{2}}$.
For the case of $|t-s|\ge(\frac{3}{4})^{\frac{\alpha+\gamma}{2}}$, by using Theorem \ref{wellpose}, we have
\begin{align*}
\ee|\bar{u}(t, x)-\bar{u}(s, x)|^{2}\le 4\sup_{(r,z)\in[0,T]\times \br^d}\ee|\bar{u}(r,z)|^2\le C|t-s|^{2\alpha_2}.
\end{align*}

\textbf{Lower bound}:
To obtain the lower bound, we introduce  two random fields $U= \{U(r,z)\}_{(r,z)\in\br_+\times \br^d}$ and $Y=\{Y(r,z)\}_{(r,z)\in\br_+\times \br^d}$ defined by
\begin{align*}
U(r,z)=& \int_{-\infty}^{0} \int_{\mathbb{R}^{d}}(G(r-\tau, z-\xi)-G(-\tau, z-\xi)) W^{H,\beta}(\ud \tau, \ud \xi) \\
&+\int_{0}^{r} \int_{\mathbb{R}^{d}} G(r-\tau, z-
\xi) W^{H,\beta}(\ud \tau, \ud \xi)\\
=&\int_{\mathbb{R}} \int_{\mathbb{R}^{d}}\left(G\left((r-\tau)_{+}, z-\xi\right)-G\left((-\tau)_{+}, z-\xi\right)\right) W^{H,\beta}(\ud \tau, \ud \xi)
\end{align*}
and
\begin{align*}
Y(r,z)=& \int_{-\infty}^{0} \int_{\mathbb{R}^{d}}(G(r-\tau, z-\xi)-G(-\tau, z-\xi)) W^{H,\beta}(\ud \tau, \ud \xi), 
\end{align*}
where $G\left(r, \cdot\right)\equiv 0$ whenever $r < 0$.
Clearly, $\bar{u}(r, z)=U(r, z)-Y(r, z)$ for all $r \ge 0$ and $z \in \mathbb{R}^{d}$. 

For the random field $U$, we have
\begin{align*}
&~\ee|U(t, x)-U(s, x)|^{2}\\
=&~\alpha_{H} \int_{\mathbb{R}} \int_{\mathbb{R}}|r-w|^{2 H-2} \ud r \ud w \int_{\mathbb{R}^{d}} \mathcal{F}\left(G((t-r)_{+},x-\cdot)-G((s-r)_{+},x-\cdot)\right)(\xi)\\
&~\qquad\qquad\times \overline{\mathcal{F}\left(G((t-w)_{+},x-\cdot)-G((s-w)_{+},x-\cdot)\right)(\xi)} \mu(\ud \xi)\\
=&~\alpha_{H}\int_{\mathbb{R}} \int_{\mathbb{R}}\ud r \ud w |r-w|^{2H-2} \int_{\mathbb{R}^{d}} \mu(\ud \xi) \varphi(r,\xi)\varphi(w,\xi)
\end{align*}
with $\varphi(r,\xi):=e^{-(t-r)\Psi(\xi)} \mathbb{1}_{\{r<t\}}-e^{-(s-r)\Psi(\xi)}\mathbb{1}_{\{r<s\}}.$
The Fourier transform of $\varphi(\cdot,\xi)$ is
\begin{align*}
\mathcal{F} \varphi(\cdot,\xi)(\tau)=\frac{e^{-\mathrm{i}  t \tau}-e^{-\mathrm{i} s \tau}}{\Psi(\xi)-\mathrm{i} \tau}.
\end{align*}
It follows from the Plancherel theorem that
$$
\alpha_{H}\int_{\mathbb{R}} \int_{\mathbb{R}} |r-w|^{2 H-2} \varphi(r,\xi) \varphi(w,\xi)\ud r \ud w=\beta_{H} \int_{\mathbb{R}}|\mathcal{F} \varphi(\cdot,\xi)(\tau)|^{2}|\tau|^{1-2 H} \ud \tau
$$
with $\beta_{H}=\frac{\alpha_{H}\Gamma(H-1/2)}{4^{1-H}\sqrt\pi\Gamma(1-H)}$. Hence, 
\begin{align*}
&~\ee|U(t, x)-U(s, x)|^{2}\\
=&~\beta_{H} \int_{\mathbb{R}^{d}} \mu(\ud \xi) \int_{\mathbb{R}} |\tau|^{1-2 H} \frac{2[1-\cos ((t-s) \tau)]}{\tau^{2}+\Psi(\xi)^{2}} \ud \tau\\
\ge&~2\beta_{H}  \int_{|\tau|>|t-s|^{-1}} [1-\cos ((t-s) \tau)]|\tau|^{-1-2 \alpha_2} \ud \tau\int_{|\xi|>|\tau|^\frac{1}{\alpha+\gamma}}\frac{|\tau|^{2-\frac{\beta}{\alpha+\gamma}}}{\tau^2+\Psi(\xi)^{2}}\mu(\ud \xi).
\end{align*}
For the case of $|t-s|<1$, $|\tau|>|t-s|^{-1}$ along with $|\xi|>|\tau|^{\frac{1}{\alpha+\gamma}}$ ensures $1+|\xi|^2<2|\xi|^2$ since $|\xi|>|\tau|^{\frac{1}{\alpha+\gamma}}>|t-s|^{-\frac{1}{\alpha+\gamma}}>1$. Thus, when $|t-s|<1$, by the change of variables $\eta=|\tau|^{-\frac{1}{\alpha+\gamma}}\xi$ and $r=|t-s|\tau$ in turn, we derive
\begin{align}\label{loweru}
&~\ee|U(t, x)-U(s, x)|^{2}\notag\\
\ge& ~2\beta_{H}  \int_{|\tau|>|t-s|^{-1}} [1-\cos ((t-s) \tau)]|\tau|^{-1-2 \alpha_2} \ud \tau\int_{|\xi|>|\tau|^\frac{1}{\alpha+\gamma}}\frac{|\tau|^{2-\frac{\beta}{\alpha+\gamma}}}{\tau^2+2^{\alpha}|\xi|^{2(\alpha+\gamma)}}\mu(\ud \xi)\notag\\
\ge&~2\beta_{H}|t-s|^{2\alpha_2} \int_{|r|>1} [1-\cos (r)]|r|^{-1-2\alpha_2} \ud r\int_{|\eta|>1}\frac{1}{1+2^{\alpha}|\eta|^{2(\alpha+\gamma)}}\mu(\ud \eta)\notag\\
=&~C_{\gamma,\alpha,d,H}|t-s|^{2\alpha_2}.
\end{align}

For the temporal increment of $Y$, we have
\begin{align*}
&\quad\ \ee\left|Y(t,x)-Y(s,x)\right|^{2} \\
&=\mathbb{E}\left|\int_{-\infty}^{0} \int_{\mathbb{R}^{d}}\left(G(t-r, x-z)-G(s-r, x-z)\right) W^{H,\beta}(\ud r, \ud z)\right|^{2} \\
&=\alpha_{H} \int_{0}^{\infty} \int_{0}^{\infty} \ud r \ud w|r-w|^{2 H-2}\int_{\mathbb{R}^{d}} \mu(\ud \xi)\phi(r,\xi)\phi(w,\xi)
\end{align*}
with $\phi(r,\xi):=\left(e^{-(t+r)\Psi(\xi)}-e^{-(s+r)\Psi(\xi)}\right) \mathbb{1}_{\{r>0\}}$.  Then for all $t_0\le s\le t \le T$,
\begin{align}\label{uppery}
&~\ee\left|Y(t,x)-Y(s,x)\right|^{2}\notag \\
=&~\beta_{H} \int_{\br^{d}} e^{-2s\Psi(\xi)}\left|1-e^{-(t-s)\Psi(\xi)}\right|^{2} \mu(\ud \xi) \int_{\mathbb{R}} \frac{\ud \tau}{\left(|\tau|^{2}+\Psi(\xi)^2\right)|\tau|^{2 H-1}}\notag \\
\le&~\beta_{H} \int_{\br^{d}} e^{-2s\Psi(\xi)}|(t-s)\Psi(\xi)|^2\mu(\ud \xi) \bigg(\int_{|\tau|<1} \frac{\ud \tau}{\Psi(\xi)^2|\tau|^{2 H-1}}+\int_{|\tau|\ge1} \frac{\ud \tau}{|\tau|^{2 H+1}}\bigg) \notag\\
\le&~ C_{H}|t-s|^{2}\int_{\br^{d}}\left(1+\Psi(\xi)^2\right)e^{-2t_0\Psi(\xi)} \mu(\ud \xi)\notag\\
\le&~C_{H,t_0}|t-s|^{2}
\end{align}
due to
$
\mathcal{F} \phi(\cdot,\xi)(\tau)= \frac{e^{-t\Psi(\xi)}-e^{- s\Psi(\xi)}}{\Psi(\xi)+\mathrm{i} \tau}.
$

By \eqref{loweru} and \eqref{uppery}, we obtain
\begin{align}\label{lower1}
\ee\left|\bar{u}(t,x)-\bar{u}(s,x)\right|^{2}
\ge&~\frac{1}{2}\ee\left|U(t,x)-U(s,x)\right|^{2} -\ee\left|Y(t,x)-Y(s,x)\right|^{2} \notag\\
\ge&~\frac{C_{\gamma,\alpha,d,H}}{2}|t-s|^{2\alpha_2}-C_{H,t_0}|t-s|^{2} \notag\\
\ge&~\frac{C_{\gamma,\alpha,d,H}}{4}|t-s|^{2\alpha_2},
\end{align}
for $|t-s|<\varepsilon$ with $\varepsilon:=\left(\frac{C_{\gamma,\alpha,d,H}}{4C_{H,t_0}}\right)^{\frac{1}{2(1-\alpha_2)}}\wedge 1>0$. For the case of $|t-s|\ge \varepsilon$, we have
\begin{align*}
&~\ee\left|\bar{u}(t,x)-\bar{u}(s,x)\right|^{2}\\
=&~\beta_{H}\int_{\br}|\tau|^{1-2H}\ud \tau\int_{\br^d}\mu(\ud\xi)\left|\mathcal F[G(t-\cdot,x-\cdot)\mathbb 1_{[0,t]}-G(s-\cdot,x-\cdot)\mathbb 1_{[0,s]}](\tau,\xi)\right|^2.
\end{align*}
The integrand can be further bounded as
\begin{align*}
&~\left|\mathcal F[G(t-\cdot,x-\cdot)\mathbb 1_{[0,t]}-G(s-\cdot,x-\cdot)\mathbb 1_{[0,s]}](\tau,\xi)\right|^2\\
=&~\frac{\left|e^{-\mathrm{i}\tau t}-e^{-t\Psi(\xi)}-e^{-\mathrm{i}\tau s}+e^{-s\Psi(\xi)}\right|^2}{\Psi(\xi)^2+\tau^2}\ge\frac{\left|\sin(\tau t)-\sin(\tau s)\right|^2}{\Psi(\xi)^2+\tau^2}
=\frac{4\cos^2( \frac{t+s}{2}\tau )\sin^2(\frac{t-s}{2}\tau)}{\Psi(\xi)^2+\tau^2}\ge0,
\end{align*}
which implies that $\ee|\bar{u}(t,x)-\bar{u}(s,x)|^{2}=0$ if and only if $t=s$.
Since $(t,s)\mapsto\ee|\bar{u}(t,x)-\bar{u}(s,x)|^{2}$ is continuous on $[t_0,T]\times[t_0,T]$, 
$$
B_\varepsilon:=\inf_{\substack{t,s\in[t_0,T]\\|t-s|\ge \varepsilon}}\ee|\bar{u}(t,x)-\bar{u}(s,x)|^{2}>0.
$$
 Thus, it follows that for $|t-s|\ge\varepsilon$,
\begin{align}\label{lower2}
\ee\left|\bar{u}(t,x)-\bar{u}(s,x)\right|^{2}\ge B_\varepsilon T^{-2\alpha_2}|t-s|^{2\alpha_2}.
\end{align}
Combining \eqref{lower1} and \eqref{lower2}, we conclude the desired lower bound. 

For the case $x\ne y$, we have
\begin{align*}
&~\ee\left|\bar{u}(t,x)-\bar{u}(s,y)\right|^{2}\\
=&~\alpha_{H}\int_0^s\int_0^s|r-w|^{2H-2}\ud r\ud w\int_{\br^d}\mu(\ud \xi)(e^{-t\Psi(\xi)}e^{-\mathrm i \langle x-y,\xi\rangle}-e^{-s\Psi(\xi)})\\
&\quad\qquad\qquad\qquad\qquad\qquad\qquad\qquad\qquad\qquad\times (e^{-t\Psi(\xi)}e^{\mathrm i \langle x-y,\xi\rangle}-e^{-s\Psi(\xi)})e^{(r+w)\Psi(\xi)}\\
&+2\alpha_{H}\int_0^s\ud r\int_s^t\ud w|r-w|^{2H-2}\int_{\br^d}\mu(\ud \xi)(e^{-(t-r)\Psi(\xi)}e^{-\mathrm i \langle x,\xi\rangle}-e^{-(s-r)\Psi(\xi)}e^{-\mathrm i \langle y,\xi\rangle})\\
&\qquad\qquad\qquad\qquad\qquad\qquad\qquad\qquad\qquad\qquad\times e^{-(t-w)\Psi(\xi)}e^{\mathrm i \langle x,\xi\rangle}\\
&+\alpha_{H}\int_s^t\int_s^t|r-w|^{2H-2}\ud r\ud w\int_{\br^d}\mu(\ud \xi) e^{-(t-r)\Psi(\xi)}e^{-(t-w)\Psi(\xi)}\\
\ge&~\alpha_{H}\int_0^s\int_0^s|r-w|^{2H-2}\ud r\ud w\int_{\br^d}\mu(\ud \xi)(e^{-t\Psi(\xi)}-e^{-s\Psi(\xi)})^2e^{(r+w)\Psi(\xi)}\\
&+2\alpha_{H}\int_0^s\ud r\int_s^t\ud w|r-w|^{2H-2}\int_{\br^d}\mu(\ud \xi)(e^{-(t-r)\Psi(\xi)}-e^{-(s-r)\Psi(\xi)})e^{-(t-w)\Psi(\xi)}\\
&+\alpha_{H}\int_s^t\int_s^t|r-w|^{2H-2}\ud r\ud w\int_{\br^d}\mu(\ud \xi) e^{-(t-r)\Psi(\xi)}e^{-(t-w)\Psi(\xi)}\\
=&~\ee\left|\bar{u}(t,x)-\bar{u}(s,x)\right|^{2},
\end{align*}
which together with \eqref{timeH} yields \eqref{timeHS}. The proof is finished.
\end{proof}

We end this section by giving the joint regularity of $\bar{u}$. For $(t,x),(s,y)\in[t_0,T]\times[-M,M]^d$, denote by
$$\varrho\big((t, x),(s, y)\big):=q_1(|x-y|)+q_2(|t-s|)$$ with $q_2(\tau):=\tau^{\alpha_2}$ for any $\tau\ge 0$. From Propositions \ref{srs} and \ref{srt}, one can obtain the following result.

\begin{corollary}\label{joint}
Under Assumptions  \ref{dalang-condition} and \ref{noise},  there exist two positive constants $C_{5}$ and $C_{6}$ such that 
\begin{align}\label{jointreg}
C_5\varrho \big((t, x),(s, y)\big)\le \|\bar{u}(t,x)-\bar{u}(s,y)\|_2\le C_6\varrho \big((t, x),(s, y)\big),
\end{align}
 for all $s,t\in[t_0,T]$ and $x,y\in[-M,M]^d$.
\end{corollary}
\begin{proof}
The upper bound is a direct application of the triangle inequality 
and the upper bounds in \eqref{spaceH} and \eqref{timeH}.

To prove the lower bound, we consider the foolowing two cases.

{\bf Case 1:} $|t-s|^{\alpha_2} \leq \frac{
C_1}{2C_4}\left(\ln \frac{C_{d,M}}{|x-y|}\right)^{\frac{1}{2}\mathbb 1_{\{\alpha_1=1\}}}|x-y|^{\alpha_1 \wedge 1}$.

Applying the triangle inequality and then, using the lower bound in \eqref{spaceH} and the upper bound in \eqref{timeH} we obtain
\begin{align*}
\|\bar{u}(t, x)-\bar{u}(s, y)\|_{2}^2 & \geq \frac{1}{2}\|\bar{u}(t, x)-\bar{u}(t, y)\|_2^2-\|\bar{u}(t, y)-\bar{u}(s, y)\|_2^2 \\
& \ge \frac{C_1^2}{2}\left(\ln \frac{C_{d,M}}{|x-y|}\right)^{\mathbb 1_{\{\alpha_1=1\}}}|x-y|^{2\alpha_1 \wedge 2}-C_4^2|t-s|^{2\alpha_2} \\
& \ge \frac{C_1^2}{8}\left(\ln \frac{C_{d,M}}{|x-y|}\right)^{\mathbb 1_{\{\alpha_1=1\}}}|x-y|^{2\alpha_1 \wedge 2}+\frac{C_4^2}{2}|t-s|^{2\alpha_2} \\
&\ge \frac{C_1^2\wedge 4C_4^2}{16}\varrho \big((t, x),(s, y)\big)^2.
\end{align*}

{\bf Case 2:} $|t-s|^{\alpha_2} >\frac{
C_1}{2C_4}\left(\ln \frac{C_{d,M}}{|x-y|}\right)^{\frac{1}{2}\mathbb 1_{\{\alpha_1=1\}}}|x-y|^{\alpha_1 \wedge 1}$.

It follows from \eqref{timeHS} that for any $x,y\in[-M,M]^d$ and $s,t\in[t_0,T]$
\begin{align*}
\|\bar{u}(t, x)-\bar{u}(s, y)\|_2  \ge&~ C_3|t-s|^{\alpha_2}\\
\geq &~\frac{C_3}{2}|t-s|^{\alpha_2}+\frac{C_3
C_1}{4C_4}\left(\ln \frac{C_{d,M}}{|x-y|}\right)^{\frac{1}{2}\mathbb 1_{\{\alpha_1=1\}}}|x-y|^{\alpha_1 \wedge 1}\\
\geq &~ \frac{C_3}{4C_4}[(2C_4)\wedge C_1 ]\varrho \big((t, x),(s, y)\big).
\end{align*}
These finish the proof of the entire corollary.
\end{proof}

\section{Hitting properties for system of fractional kinetic equations}\label{section4}
Let $\hat{u}=\left\{\hat{u}(t, x)\right\}_{(t,x) \in [0,T]\times \br^d}$ be an $\br^n$-valued random field whose components $\hat{u}_j, j = 1,\dots,n$, are independent copies of the mild solution $u$ of the fractional kinetic equation \eqref{FKE}. In this section, we are devoted to investigating hitting probabilities and the polarity of points in the critical dimension for $\hat u$. 
\subsection{Hitting probability} 
Throughout this part, we use the notation
$$
\sigma_{t, x}^2=\ee|\bar{u}(t, x)|^{2}, ~\rho_{(t, x),(s, y)}=\frac{\ee[\bar{u}(t, x) \bar{u}(s, y)]}{\sigma_{t, x} \sigma_{s, y}},\quad s, t \in[t_0, T], x, y \in \br^{d}.
$$
\begin{lemma}\label{2ndorder}
Under Assumptions  \ref{dalang-condition} and \ref{noise}, the following properties hold:
\begin{itemize}
\item[(i)] There exists a constant  $c_{d, t_0,T}$  such that for all  $s, t \in[t_0, T], x, y \in [-M,M]^d$,
\begin{align}
\left|\sigma_{t, x}^{2}-\sigma_{s, y}^{2}\right| \leq c_{d,t_0, T}\|\bar{u}(t, x)-\bar{u}(s, y)\|_2^{\frac{1}{\alpha_2}}.
\end{align}
\item[(ii)] There exist constants  $0<C_{d, t_{0},T}<C_{d, T}$ such that for any  $(t, x) \in   \left[t_{0}, T\right] \times \br^d$,
\begin{align}\label{boundsigma}
 C_{d, t_{0},T} \leq \sigma_{t, x}^{2} \leq C_{d, T}.
\end{align}
\item[(iii)] For any $ (t, x),(s, y) \in\left[t_{0}, T\right] \times [-M,M]^d$ with $(t, x) \neq(s, y)$,
\begin{align}\label{corr}
\rho_{(t, x),(s, y)}<1.
\end{align}
\end{itemize}
\end{lemma}
\begin{proof}
(i) By the definition of $\sigma_{t, x}^2$ and \eqref{Jt}, we have 
\begin{align}\label{sigam-2}
|\sigma_{t, x}^2-\sigma_{s, y}^2|&=\alpha_{H}\int_{0}^{t} \int_{0}^{t}\ud r \ud w|r-w|^{2 H-2} \int_{\br^{d}} e^{-(r+w)\Psi(\xi)}\mu(\ud \xi)  \notag\\
&\quad~-\alpha_{H}\int_{0}^{s} \int_{0}^{s}\ud r \ud w|r-w|^{2 H-2} \int_{\br^{d}}e^{-(r+w)\Psi(\xi)}\mu(\ud \xi) \notag\\
&=2\alpha_{H}\sigma_1(s,t)+\alpha_{H}\sigma_2(s,t),
\end{align}
where 
\begin{align*}
 \sigma_1(s,t):=\int_{0}^{s} \int_{s}^{t}\ud r \ud w|r-w|^{2 H-2} \int_{\br^{d}}e^{-(r+w)\Psi(\xi)}\mu(\ud \xi),\\
 \sigma_2(s,t):=\int_{s}^{t} \int_{s}^{t}\ud r \ud w|r-w|^{2 H-2} \int_{\br^{d}}e^{-(r+w)\Psi(\xi)}\mu(\ud \xi).
\end{align*}
Since $H\in(\frac{1}{2},1)$ and $\int_{\br^{d}}e^{-c\Psi(\xi)}\mu(\ud \xi)<\infty$ for any $c>0$, we obtain
\begin{align*}
 \sigma_1(s,t)&\le \int_{0}^{s} \ud r\int_{s}^{t} (w-r)^{2 H-2} \ud w\int_{\br^{d}}e^{-t_0\Psi(\xi)}\mu(\ud \xi) \notag\\
&\le  C\frac{t^{2H}-(t-s)^{2H}-s^{2H}}{2H(2H-1)}\le C|t-s|,
\end{align*}
due to $t^{2H}-s^{2H}\le 2HT^{2H-1}(t-s)$ and $t-s<T$. Similarly,
\begin{align}\label{sigam2}
 \sigma_2(s,t)\le |t-s|^{2H}\int_{0}^{1} \int_{0}^{1}\ud r \ud w|r-w|^{2 H-2} \int_{\br^{d}}   e^{-2t_0\Psi(\xi)} \mu(\ud \xi)\le C|t-s|^{2H}.
\end{align}
By \eqref{sigam-2} and \eqref{sigam2}, we obtain $|\sigma_{t, x}^2-\sigma_{s, y}^2|\le C|t-s|\le \tilde{C} \varrho \big((t, x),(s, y)\big)^{\frac{1}{\alpha_2}}.$ This together with Corollary \ref{joint} yields the desired result.

(ii) For any $t\in[t_0,T]$ and $x\in\br^d$, it follows from Theorem \ref{wellpose} and \eqref{Jt} that
\begin{align}
\sigma_{t, x}^{2} \le&~\sup _{(r,z ) \in[0, T]\times \mathbb{R}^{d}} \ee\left|\bar{u}(r,z ) \right|^2=:C_{d, T}<\infty
\end{align}
and
\begin{align}\label{nondeg}\notag
\sigma_{t, x}^{2} \ge&~
 \alpha_{H}\int_{0}^{t_0} \int_{0}^{t_0}\ud r \ud w|r-w|^{2 H-2} \int_{\br^{d}}\mu(\ud \xi)   \exp \left(-2T\Psi(\xi)\right)\\
=&~t_0^{2H}\int_{\br^{d}}  \exp \left(-2T\Psi(\xi)\right)\mu(\ud \xi)=:C_{d, t_{0},T}>0.
\end{align}

(iii) Assume by contradiction that $\rho_{(t, x),(s, y)}=1$ for some $(t, x) \neq(s, y)$. The sufficient and necessary condition under which the H\"older inequality becomes an equality implies that there exists $\lambda \in \br\setminus \{0\}$ such that $\bar{u}(t, x)=$ $\lambda \bar{u}(s, y)$ a.s., and in particular,
\begin{align}\label{contradiction}
\ee|\bar{u}(t, x)-\lambda \bar{u}(s, y)|^2=0 .
\end{align}
The  proof of \eqref{corr} is divided into two cases.

{\bf Case 1: $s<t$.} 
The left-hand side of \eqref{contradiction} is equal to
\begin{align*}
&~\beta_{H}\int_{\br}|\tau|^{1-2H}\ud\tau\int_{\br^d}\mu(\ud\xi)\left|\frac{e^{-\mathrm i \langle\xi   ,x\rangle}(e^{-\mathrm{i} \tau t}-e^{-t\Psi(\xi)})-\lambda e^{-\mathrm i \langle\xi ,y\rangle}(e^{-\mathrm{i} \tau s}-e^{-s\Psi(\xi)})}{\Psi(\xi)-\mathrm{i} \tau}\right|^{2}\\
=&~\beta_{H}\int_{\br}|\tau|^{1-2H}\ud\tau\int_{\br^d}\mu(\ud\xi)\frac{\left|e^{-\mathrm i \langle\xi   ,x-y\rangle}(e^{-\mathrm{i} \tau t}-e^{-t\Psi(\xi)})-\lambda (e^{-\mathrm{i} \tau s}-e^{-s\Psi(\xi)})\right|^{2}}{\Psi(\xi)^2+ \tau^2}.
\end{align*}
The integrand vanishes only when the continuous function 
$$f_{\lambda}(\xi,\tau)=e^{-\mathrm i \langle\xi   ,x-y\rangle}(e^{-\mathrm{i} \tau t}-e^{-t\Psi(\xi)})-\lambda (e^{-\mathrm{i} \tau s}-e^{-s\Psi(\xi)})\equiv 0$$
for any $\tau\in \br$ and $\xi\in\br^d$. If so, we have $ f_{\lambda}(\eta,0)=1-\lambda+\lambda e^{-s\Psi(\eta)}-e^{-t\Psi(\eta)}\equiv 0$ for any $\eta\in\{x-y\}^{\bot }\setminus \{0\}$, where $\{x-y\}^{\bot }$ is the orthogonal complement of $\text{span}\{x-y\}$. Let $|\eta|\to\infty$, we obtain $\lambda=1$. But $e^{-s\Psi(\eta)}-e^{-t\Psi(\eta)}>0$ for all $\eta\in\{x-y\}^{\bot }\setminus \{0\}$. This leads to a contradiction.

{\bf Case 2: $s=t,~ x \neq y.$}  The left-hand side of \eqref{contradiction} is equal to 
\begin{align*}
\alpha_{H}\int_{0}^{t}\int_{0}^{t}|r-w|^{2H-2} \ud r\ud w \int_{\br^d} \mu(\ud \xi)\big|\mathrm{e}^{-\mathrm{i} \langle\xi,  x\rangle}-\lambda \mathrm{e}^{-\mathrm{i} \langle\xi , y\rangle}\big|^{2}e^{-(t-r)\Psi(\xi)} e^{-(t-w)\Psi(\xi)} .
\end{align*}

If $\lambda=1$, then the integrand vanishes only on the set $\{\xi\in\mathbb R^d:\cos \big(\langle\xi ,x-y\rangle\big)=1\}$, which has  a zero Lebesgue measure.
Hence, by Assumption \ref{noise} (i), we reach a contradiction. 

If $\lambda \neq 1$, then $|e^{-\mathrm{i} \langle\xi,  x\rangle}-\lambda e^{-\mathrm{i} \langle\xi,  y\rangle}|=1+\lambda^2-2\lambda\cos(\langle\xi ,x-y\rangle)\ge (1-|\lambda|)^2$, and hence
\begin{align*}
 &\alpha_H\int_{0}^{t}\int_{0}^{t}|r-w|^{2H-2} \ud r\ud w \int_{\br^d} \mu(\ud \xi)\big|e^{-\mathrm{i} \langle\xi,  x\rangle}-\lambda e^{-\mathrm{i} \langle\xi,  y\rangle}\big|^{2}e^{-(t-r)\Psi(\xi)} e^{-(t-w)\Psi(\xi)} \\
\ge&(1-|\lambda|)^2 \int_{\br^d} N_{t_0}(\xi)\mu(\ud \xi) >0.
 \end{align*}
Thus, we also get a contradiction in this case. 
\end{proof}

One of main results of this section is the following theorem,
which gives the lower and upper bounds for the hitting probability of the random field $\hat u$. We first introduce some notations. For $\tau \in \mathbb{R}_{+},x\in\br^d$, let
$$
g_{q}(\tau):=\tau^n \check{q_1}(\tau)^{-d}\check{q_2}(\tau)^{-1},\quad \mathfrak{g}_q(x):=g_{q}(|x|)^{-1}.
$$
We refer to \cite[Section 5]{MR4333508} for specific properties of these functions.

\begin{theorem}\label{hittingprobability}
Under Assumptions  \ref{dalang-condition} and \ref{noise}, for any $n\ne Q:=\alpha_2^{-1}+d(1 \wedge\alpha_1)^{-1}$,  the hitting probabilities of the $n$-dimensional random field $\hat u$ satisfy the following bounds.\\
(i) There exists a constant $C:=C_{\alpha,\gamma,H,t_0, T,M, n, d}$ such that for any $A \in \mathscr{B}\left(\br^n\right) $
\begin{align}\label{hitupper}
\bp(\hat{u}( [t_0,T]\times[-M,M]^d)\cap A \neq \emptyset) \le C \mathcal{H}_{g_q}(A).
\end{align}
(ii) For any bounded Borel set $A \in \mathscr{B}\left(\br^n\right)$, i.e., $A \subset [-N,N]^n$ for some $N>0$, there exists a constant $C:=C_{\alpha,\gamma,H,t_0, T,M, N, n, d}$ such that
\begin{align}\label{hitlower}
\bp(\hat{u}( [t_0,T]\times[-M,M]^d) \cap A \neq \emptyset) \ge 
C \operatorname{Cap}_{\mathfrak{g}_q}(A).
\end{align}
\end{theorem}
\begin{proof}
The proof is divided into two cases.

{\bf Case 1: $n>Q$.} Based on Corollary \ref{joint} and Lemma \ref{2ndorder}, the upper bound \eqref{hitupper} 
and the lower bound \eqref{hitlower} follows from \cite[Theorem 3.3]{MR4333508} and \cite[Theorem 3.5]{MR4333508}, respectively.

{\bf Case 2: $n< Q$.} It follows from the proof of \cite[Lemma 5.1]{MR4333508} that $g_q(0)=\lim_{\tau\downarrow  0} g_q(\tau) = \infty$ and then $\mathcal H_{g_q} (A) = \infty$. Hence, \eqref{hitupper} holds without any further information.

The lower bound \eqref{hitlower} can be obtained by the same strategy as in the proof of \cite[Theorem 2.1]{MR2759182}. To avoid the cost of a heavier notation, we write $I:=[t_0,T]$ and $J:=[-M,M]^d$. For $\varepsilon \in(0,1)$ and $z \in A $, we denote by $A^{(\varepsilon)}$ the $\varepsilon$-enlargement of $A$ and define
$$
J_{\varepsilon}(z)=\frac{1}{(2 \varepsilon)^n} \int_I \ud s \int_J  \ud y ~\mathbb1_{B(z,\varepsilon)}(\hat{u}(s, y)).
$$
Since $\left\{J_{\varepsilon}(z)>0\right\} \subset\left\{\hat{u}(I\times J) \cap A^{(\varepsilon)} \neq \emptyset\right\}$, it suffices to estimate the lower bound of $\bp\left(J_{\varepsilon}(z)>0\right)$. 
Because of \eqref{boundsigma}, the density $p_{t,x}(z)$ of $\hat{u}(t, x)$ is bounded uniformly on $(t, x) \in I\times J$. This yields $\ee\left[J_{\varepsilon}(z)\right]>C$.

From Lemma \ref{2ndorder} and \cite[Proposition 3.1]{MR2759182}, we deduce that the density $p_{s, y ; t, x}\left(z_{1}, z_{2}\right) $ of $(\hat{u}(s, y), \hat{u}(t, x))$ satisfies that for $z_{1}, z_{2} \in A^{(\varepsilon)},$
$$
p_{s, y ; t, x}\left(z_{1}, z_{2}\right) \leq \frac{C}{[\varrho((s, y),(t, x))]^n} \exp \left(-\frac{c\left|z_{1}-z_{2}\right|^{2}}{[\varrho((s, y),(t, x))]^{2}}\right).
$$
Consequently, we have
\begin{align*}
\ee\left[\left|J_{\varepsilon}(z)\right|^{2}\right] =&~ \frac{1}{(2 \varepsilon)^{2n}}\int_{I \times J} \ud s \ud y \int_{I \times J} \ud t \ud x ~\ee [ \mathbb1_{B(z,\varepsilon)}(\hat{u}(s, y)) \mathbb1_{B(z,\varepsilon)}(\hat{u}(t, x))]\\
\leq &~C \int_{I \times J} \ud s \ud y \int_{I \times J} \ud t \ud x~[\varrho((s, y),(t, x))]^{-n}\\
\leq&~C \int_{I \times J} \ud s \ud y \int_{I \times J} \ud t \ud x~\left[|t-s|^{\alpha_2}+|x-y|^{\alpha_1\wedge 1}\right]^{-n},
\end{align*}
where in the last line we used the fact that $\ln \frac{C_{d,M}}{|x-y|} \geq 1$. Let $C_{T,M,d} >(2T)\vee(2\sqrt d M)$. Fix $(t,x)\in I\times J$, after the change of variables $(r,w)= (t-s,x-y)$ and by using polar coordinates, we easily obtain
\begin{align*}
&~\int_{I \times J} \ud s \ud y \int_{I \times J} \ud t \ud x~\left[|t-s|^{\alpha_2}+|x-y|^{\alpha_1\wedge 1}\right]^{-n}\\
\le &~C\int_{0}^{C_{T,M,d}} \ud r\int_0^{C_{T,M,d}}\ud \rho~\left[r^{\alpha_2}+\rho^{\alpha_1\wedge 1}\right]^{-n}\rho^{d-1}\\
= &~C\int_{0}^{C_{T,M,d}} \ud \rho\int_0^{C_{T,M,d}\rho^{-\frac{\alpha_1\wedge 1}{\alpha_2}}}\ud r~\left[1+r^{\alpha_2}\right]^{-n}\rho^{d-1+\frac{\alpha_1\wedge1}{\alpha_2}-n(\alpha_1\wedge1)},
\end{align*}
where in the last step we used the change of variable $r\to r\rho^{\frac{\alpha_1\wedge1}{\alpha_2}}$. Then Fubini's theorem and the relation $n<Q$ yield
\begin{align*}
\ee\left[\left|J_{\varepsilon}(z)\right|^{2}\right]\le &~C\int_{C}^{\infty} \ud r\int_0^{(\frac{C_{T,M,d}}{r})^{\frac{\alpha_2}{\alpha_1\wedge 1}}}\ud \rho~\left[1+r^{\alpha_2}\right]^{-n}\rho^{d-1+\frac{\alpha_1\wedge1}{\alpha_2}-n(\alpha_1\wedge1)}\\
\le &~C\int_{C}^{\infty} \left[1+r^{\alpha_2}\right]^{-n}r^{-\alpha_2(Q-n)}\ud r\\
\le &~C\int_{C}^{\infty} r^{-\frac{\alpha_2d}{\alpha_1\wedge1}-1}\ud r<\infty.
\end{align*}
Using the Paley--Zygmund inequality (see \cite[Chapter 3, Lemma 1.4.1]{MR1914748}) leads to
$$
\bp\left(J_{\varepsilon}(z)>0\right)\ge \frac{(\ee\left[J_{\varepsilon}(z)\right])^2}{\ee\left[\left|J_{\varepsilon}(z)\right|^{2}\right]}>C.
$$
Since $\mathfrak{g}_q(0)=\left[g_{q}(|0|)\right]^{-1}=0$, $\operatorname{Cap}_{\mathfrak{g}_q}(A)=1$. This yields the lower bound \eqref{hitlower}.
\end{proof}

Based on Theorem \ref{hittingprobability}, we have the following corollary for the polarity of points when $n\neq Q$.
\begin{corollary}\label{polaritypoints}
Under Assumptions  \ref{dalang-condition} and \ref{noise},  points $z \in \br^n$  are polar for $\hat{u}$ if $n>Q$ and are non-polar if $n<Q$.
\end{corollary}
\begin{proof}
If $n>Q$, the definition of the $g_{q}$-Hausdorff measure implies that $\mathcal{H}_{g_{q}}(\{z\})=0$. Hence, the polarity of $\{z\}$ follows from \eqref{hitupper}.
If $n<Q$, it follows from $\operatorname{Cap}_{\mathfrak{g}_q}(\{z\})=1$ and \eqref{hitlower} that $\{z\}$ is nonpolar. 
\end{proof}

\subsection{Polarity of points in the critical dimension}\label{points-critical}

In this part, we prove that points are polar for $\hat u$ with vanishing initial datum in the critical dimension for the case of $\alpha_1\in(0,1)$. The issue of polarity of points for the case of $\alpha_1\ge1$ is still an unsolved problem. 
To begin with, we introduce the harmonizable representation for the solution of \eqref{FKE}.
Let $W$ be a space-time white noise in $\mathbb R^{d+1}$, and $\widehat W$ be the Fourier transform of $W$ such that for any $g\in L^2(\mathbb R^{d+1},\mathbb C)$ with $g(-\xi)=\overline{g(\xi)}$, 
$$\int_{\mathbb R^{d+1}}g(\xi)\widehat W(\ud \xi):=\int_{\mathbb R^{d+1}}\mathcal Fg(z)W(\ud z);$$
see \cite[Definition 2.1.16]{MR3088856} for more details about the Fourier transform of a white noise. Introduce a centered Gaussian random field $v=\{v(t,x)\}_{(t,x)\in\br_+\times\mathbb R^d}$ with
\begin{align*}
v(t, x)= \sqrt{\beta_{H}}\int_{\br} \int_{\br^d}  |\tau|^{\frac{1}{2}-H} \Upsilon(\eta)^{\frac{1}{2}}\frac{e^{-\mathrm i \langle\eta, x\rangle}(e^{-\mathrm{i} \tau t}-e^{-t\Psi(\eta)})}{\Psi(\eta)-\mathrm{i} \tau} \widehat{W}(\ud \tau, \ud \eta).
\end{align*}
Then the covariance function of $v$ coincides with that of $u$ since
 \begin{align*}
\ee[v(t,x)v(s,y)]
= \beta_{H}\int_\br|\tau|^{1-2H}\ud \tau\int_{\br^d}\mu(\ud\eta)e^{-\mathrm{i}\langle\eta,x-y\rangle}\frac{(e^{-\mathrm{i}\tau t}-e^{-t\Psi(\eta)})(e^{\mathrm{i}\tau s}-e^{-s\Psi(\eta)})}{\Psi(\eta)^2+\tau^2},
\end{align*}
which implies that the solution $u$ of \eqref{FKE} has the same law as $v$.

The following lemma indicates that there is a Gaussian random field $\{v(A, t, x); A \in \mathscr{B}\left(\br_{+}\right),(t, x) \in \br_{+} \times \br^d\}$ which gives a good approximation for the random field $v$.
\begin{lemma}\label{approximationfield}
Assume that Assumption \ref{noise} holds and $\alpha_1\in(0,1).$ Then for any $A\in\mathscr{B}(\br_+)$, the random field
$$
v(A, t, x):=\sqrt{\beta_{H}}\iint_{|\tau|^{\alpha_{2}}\vee  |\eta|^{\alpha_{1}}\in A}|\tau|^{\frac{1}{2}-H}  \Upsilon(\eta)^{\frac{1}{2}} \frac{e^{-\mathrm i \langle\eta, x\rangle}(e^{-\mathrm{i} \tau t}-e^{-t\Psi(\eta)})}{\Psi(\eta)-\mathrm{i} \tau} \widehat{W}(\ud \tau, \ud \eta)
$$
fulfills the following properties:
\item[(i)]  For all $(t,x) \in[t_0,T]\times[-M,M]^d, A \mapsto v(A, t, x)$ is a real-valued white noise based on a measure $\nu$ with
$$
\nu(A)=\beta_{H}\iint_{|\tau|^{\alpha_{2}}\vee  |\eta|^{\alpha_{1}}\in A}\frac{|e^{-\mathrm{i} \tau t}-e^{-t\Psi(\eta)}|^2}{\Psi(\eta)^2+\tau^2} |\tau|^{1-2H}  \Upsilon(\eta) \ud \tau \ud \eta, \quad A\in\mathscr B(\br_+).
$$
Besides, $v(\br_+, t, x)=v(t,x)$ as well as $v(A, \cdot)$ and $v(B, \cdot)$ are independent whenever $A$ and $B$ are disjoint.
\item[(ii)] There is a universal constant $C$ such that for all $0 \le a \le b$ and $\left(s, y\right), (t, x) \in \br_{+} \times \br^d$,
\begin{align*}
\| v([a, b), t, x)-v(t, x)-v([a, b), s, y)+v(s, y) \|_2\le C\Big[a^{\gamma_{2}}\left|t-s\right|+a^{\gamma_1} \sum_{i=1}^{d}\left|x_i-y_i\right|+b^{-1}\Big]
\end{align*}
and
\begin{align*}
\|v([0, a], t, x)-v([0, a], s, y)\|_2\le C\Big(|t-s|+\sum_{i=1}^d|x_i-y_i|\Big),
\end{align*}
where $\gamma_{1}:=\alpha_{1}^{-1}-1$ and $\gamma_{2}:=\alpha_{2}^{-1}-1$.
\end{lemma}
\begin{proof}
(i) It is easy to check that $\nu$ is a measure on $\br_+$ and for any $(t,x) \in[t_0,T]\times[-M,M]^d,$ $\{v(A, t, x)\}_{A\in\mathscr B(\br_+)}$ is a centered Gaussian random field with covariance function
$$\ee[v(A, t, x)v(B, t, x)]=\nu(A\cap B),$$
and if $A \cap B=\emptyset$, then $v(A \cup B,t,x)=v(A,t,x)+v(B,t,x)$ as well as $v(A,t,x)$ and $v(B,t,x)$ are independent. 

(ii) Let
\begin{align*}
v_{1}(a, t, x):=&\sqrt{\beta_{H}}\iint_{|\tau|^{\alpha_{2}}\vee  |\eta|^{\alpha_{1}}<a}|\tau|^{\frac{1}{2}-H} \Upsilon(\eta)^{\frac{1}{2}}\frac{e^{-\mathrm i \langle\eta, x\rangle}(e^{-\mathrm{i} \tau t}-e^{-t\Psi(\eta)})}{\Psi(\eta)-\mathrm{i} \tau} \widehat{W}(\ud \tau, \ud \eta) ,\\
v_{2}(b, t, x):=&\sqrt{\beta_{H}}\iint_{|\tau|^{\alpha_{2}}\vee  |\eta|^{\alpha_{1}}\ge b}|\tau|^{\frac{1}{2}-H}  \Upsilon(\eta)^{\frac{1}{2}}\frac{e^{-\mathrm i \langle\eta, x\rangle}(e^{-\mathrm{i} \tau t}-e^{-t\Psi(\eta)})}{\Psi(\eta)-\mathrm{i} \tau} \widehat{W}(\ud \tau, \ud \eta) .
\end{align*}
Then we have
$$
v([a, b), t, x)-v(t, x)-v([a, b) ,s, y)+v(s, y)=v_{1}(a, s, y)-v_{1}(a, t, x)+v_{2}(b, s, y)-v_{2}(b, t, x).
$$
Set
\begin{align*}
f_{1}\left(a, t, x, s, y\right)&=\ee|v_{1}(a, t, x)-v_{1}(a, s,y)|^{2} ,\\
f_{2}\left(b, t, x, s, y\right)&=\ee|v_{2}(b, t, x)-v_{2}(b, s, y)|^2.
\end{align*}
We shall estimate these two quantities separately. First, set
$$
D_{1}(a):=\left\{(\tau, \eta) \in \br \times \br^d: |\tau|^{\alpha_{2}}\vee|\eta|^{\alpha_{1}}<a\right\}.
$$
Then
\begin{align*}
&~f_{1}(a, t,x, s, y) \\
=&~\beta_{H}\iint_{D_{1}(a)}|\tau|^{1-2H}  \frac{\left|e^{-\mathrm i \langle\eta ,x-y\rangle}(e^{-\mathrm{i} \tau t}-e^{-t\Psi(\eta)})-(e^{-\mathrm{i} \tau s}-e^{-s\Psi(\eta)})\right|^2}{\Psi(\eta)^2+ \tau^2} \ud \tau \mu(\ud \eta) \\
=&~\beta_{H}\iint_{D_{1}(a)}|\tau|^{1-2H} \frac{\varphi_{1}(t, x, \tau, \eta)^{2}+\varphi_{2}(t, x, \tau, \eta)^{2}}{\Psi(\eta)^2+ \tau^2}\ud \tau \mu(\ud \eta),
\end{align*}
where
\begin{align*}
\varphi_{1}(t, x, \tau, \eta):=& \cos (\langle\eta,x-y\rangle+\tau t)- e^{-t\Psi(\eta)}\cos \left(\langle\eta,x-y\rangle\right)-\cos \left(\tau s\right)+e^{-s\Psi(\eta)}, \\
\varphi_{2}(t, x, \tau, \eta):=&-\sin (\langle \eta,x-y\rangle+\tau t)+\sin( \tau s)+e^{-t \Psi(\eta)} \sin \left(\langle\eta,x-y\rangle\right).
\end{align*}
Observe that $\varphi_{1}\left(s, y, \tau, \eta\right)=0=\varphi_{2}\left(s, y, \tau, \eta\right)$, and
\begin{align*}
\partial_t \varphi_{1}(t, x, \tau, \eta)&=-\tau \sin (\langle\eta,x-y\rangle+\tau t)+\Psi(\eta)e^{-t\Psi(\eta)}\cos(\langle\eta,x-y\rangle),\\
\nabla_x \varphi_{1}(t, x, \tau, \eta)&=-\eta \sin (\langle\eta,x-y\rangle+\tau t)+\eta e^{-t\Psi(\eta)} \sin (\langle\eta,x-y\rangle), \\
\partial_t \varphi_{2}(t, x, \tau, \eta)&=-\tau \cos (\langle \eta,x-y\rangle+\tau t)-\Psi(\eta)e^{-t \Psi(\eta)} \sin \left(\langle\eta,x-y\rangle\right), \\
\nabla_x\varphi_{2}(t, x, \tau, \eta)&=-\eta \cos (\langle\eta,x-y\rangle+\tau t)+\eta e^{-t\Psi(\eta)}  \cos (\langle\eta,x-y\rangle).
\end{align*}
Therefore, for $i=1,2$, 
$$
\left|\partial_t \varphi_i\right| \le|\tau|+\Psi(\eta), \quad\left|\nabla_x \varphi_i \right| \le 2|\eta|,
$$
and the mean value theorem implies that
$$
\left|\varphi_i(t, x, \tau, \eta)\right| \le\left(|\tau|+\Psi(\eta)\right)\left|t-s\right|+2|\eta|\left|x-y\right|.
$$
So
\begin{align}\label{f_1}
f_{1}\left(a, t, x, s, y\right) \le & ~\beta_{H}\iint_{D_{1}(a)}|\tau|^{1-2H}\frac{8\left(\tau^{2}+\Psi(\eta)^{2}\right)\left|t-s\right|^{2}+16|\eta|^{2}\left|x-y\right|^{2}}  {\Psi(\eta)^{2}+\tau^{2}} \ud \tau \mu(\ud \eta)\notag\\
\le &~ 8\beta_{H}|t-s|^{2} \iint_{D_{1}(a)}|\tau|^{1-2H} \ud \tau \mu(\ud \eta) \notag\\
&~+16\beta_{H}\left|x-y\right|^{2} \iint_{D_{1}(a)}\frac{|\eta|^2|\tau|^{1-2H} }{\Psi(\eta)^{2}+\tau^{2}} \ud \tau \mu(\ud \eta).
\end{align}
Using the change of variables $r=a^{-\frac{1}{\alpha_2}}\tau$ and $\xi=a^{-\frac{1}{\alpha_1}}\eta$, we obtain
\begin{align*}
\iint_{D_{1}(a)}|\tau|^{1-2H} \ud \tau \mu(\ud \eta) 
=a^{2\gamma_2}\iint_{D_{1}(1)}|r|^{1-2H} \ud r \mu(\ud \xi)
\le C_Ha^{2\gamma_2}.
\end{align*}
Similarly,  by the change of variables $r=a^{-\frac{1}{\alpha_2}}\tau$, $\xi=a^{-\frac{1}{\alpha_1}}\eta$ and $\zeta=\xi |r|^{-\frac{1}{\alpha+\gamma}}$, we have 
\begin{align}\label{eq:f1-2}\notag
&\quad\ \iint_{D_{1}(a)}\frac{|\eta|^2|\tau|^{1-2H} }{\Psi(\eta)^{2}+\tau^{2}} \ud \tau \mu(\ud \eta)\\\notag
&\le\iint_{D_{1}(a)}\frac{|\eta|^2|\tau|^{1-2H}} {|\eta|^{2(\alpha+\gamma)}+\tau^{2}} \ud \tau \mu(\ud \eta)\\\notag
&=a^{2\gamma_1}\iint_{D_{1}(1)}\frac{|\xi|^2 |r|^{1-2H}}{|\xi|^{2(\alpha+\gamma)}+r^{2}} \ud r \mu(\ud \xi)\\
&=2a^{2\gamma_1}\int_0^1 r^{\frac{2}{\alpha+\gamma}-2\alpha_2-1}\ud r\int_{|\zeta|<r^{-\frac{1}{\alpha+\gamma}}}\frac{ |\zeta|^2 }{|\zeta|^{2(\alpha+\gamma)}+1}  \mu(\ud \zeta).
\end{align}
In order to deal with the integral w.r.t. $\zeta$, we separate the set $\{|\zeta|<r^{-\frac{1}{\alpha+\gamma}}\}$ into $\{|\zeta|\le1\}$ and $\{1<|\zeta|<r^{-\frac{1}{\alpha+\gamma}}\}$, and obtain
\begin{align*}
\int_{|\zeta|<r^{-\frac{1}{\alpha+\gamma}}}\frac{ |\zeta|^2 }{|\zeta|^{2(\alpha+\gamma)}+1}  \mu(\ud \zeta)\le C+\int_{1<|\zeta|<r^{-\frac{1}{\alpha+\gamma}}}\frac{ |\zeta|^2 }{|\zeta|^{2(\alpha+\gamma)}+1}  \mu(\ud \zeta),
\end{align*}
because under Assumption \ref{dalang-condition}, the integral $\int_{|\zeta|\le 1}\mu(\ud\zeta)$ is finite. Next, we apply the polar coordinate transform and Lemma \ref{Upsilon} (iii) to derive
\begin{align*}
 \int_{1<|\zeta|<r^{-\frac{1}{\alpha+\gamma}}}\frac{ |\zeta|^2 }{|\zeta|^{2(\alpha+\gamma)}+1}  \mu(\ud \zeta)
&=\int_1^{r^{-\frac{1}{\alpha+\gamma}}}\rho^{2-2(\alpha+\gamma)}\rho^{\beta-1}\ud \rho\int_{\mathbb S^{d-1}}\Upsilon(z)\sigma(\ud z)\\
&\le\begin{cases}
Cr^{-\frac{2+\beta}{\alpha+\gamma}+2},&\text{if}\quad1-(\alpha+\gamma)+\beta/2>0,\\
C|\ln r|,&\text{if}\quad1-(\alpha+\gamma)+\beta/2=0,\\
C,&\text{if}\quad1-(\alpha+\gamma)+\beta/2<0.\end{cases}
\end{align*}
It follows from the assumption $\alpha_1\in(0,1)$ that $\frac{2}{\alpha+\gamma}-2\alpha_2-1=\frac{2-2\alpha_1}{\alpha+\gamma}-1>-1$ and $\frac{2}{\alpha+\gamma}-2\alpha_2-1-\frac{2+\beta}{\alpha+\gamma}+2=1-2H>-1$, which implies that
\begin{align*}
\int_0^1 r^{\frac{2}{\alpha+\gamma}-2\alpha_2-1}\ud r\int_{|\zeta|<r^{-\frac{1}{\alpha+\gamma}}}\frac{ |\zeta|^2 }{|\zeta|^{2(\alpha+\gamma)}+1}  \mu(\ud \zeta)<\infty.
\end{align*}
As a consequence of \eqref{f_1}--\eqref{eq:f1-2}, we derive
$$
f_{1}\left(a, t, x, s, y\right) \le C\left[a^{2 \gamma_2}\left(t-s\right)^{2}+a^{2 \gamma_1}\left|x-y\right|^{2}\right] . 
$$
Set $D_{2}(b):=\left\{(\tau, \eta)\in \br \times \br^d: |\tau|^{\alpha_{2}}\vee|\eta|^{\alpha_{1}}\ge b\right\}$.
Similarly to the estimate of $f_1$, we have
\begin{align*}
f_2\left(b, t, x, s,y\right)
=\beta_{H}\iint_{D_{2}(b)} |\tau|^{1-2H} \frac{\varphi_{1}(t, x, \tau, \eta)^{2}+\varphi_{2}(t, x, \tau, \eta)^{2}}{\Psi(\eta)^2+ \tau^2}\ud \tau \mu(\ud \eta).
\end{align*}
Observing that $\left|\varphi_{1}\right| \le 4$ and $\left|\varphi_{2}\right| \le 3$, and using the change of variables $r=b^{-\frac{1}{\alpha_2}}\tau$ and $\xi=b^{-\frac{1}{\alpha_1}}\eta$, we see that
\begin{align*}
f_{2}\left(b, t, x, s, y\right)& \le C\iint_{D_{2}(b)} \frac{|\tau|^{1-2H} }{\Psi(\eta)^2+ \tau^2}\ud \tau \mu(\ud \eta) \le C\iint_{D_{2}(b)}  \frac{|\tau|^{1-2H}}{|\eta|^{2(\gamma+\alpha)}+ \tau^2}\ud \tau \mu(\ud \eta)\\
&= C b^{-2} \iint_{D_{2}(1)} \frac{|r|^{1-2H}  }{|\xi|^{2(\gamma+\alpha)}+ r^2}\ud r \mu(\ud \xi).
\end{align*}
By decomposing the set $D_2(1)$ into $\{(r, \xi)\in \br \times \br^d: |r|^{\alpha_{2}}\ge|\xi|^{\alpha_{1}},|r|\ge1\}$ and $\{(r, \xi)\in \br \times \br^d: |r|^{\alpha_{2}}<|\xi|^{\alpha_{1}},|\xi|\ge1\}$, and using the change of variables $\eta=|r|^{-\frac{1}{\alpha+\gamma}}\xi$ and $\tau=|\xi|^{-(\alpha+\gamma)}r$,
\begin{align*}
&~\iint_{D_{2}(1)} \frac{|r|^{1-2H}  }{|\xi|^{2(\gamma+\alpha)}+ r^2}\ud r \mu(\ud \xi)\\
\le&~\int_{|r|\ge1} |r|^{-1-2H} \ud r\int_{|\xi|^{\alpha_1}\le|r|^{\alpha_2}}\mu(\ud \xi)+\int_{|\xi|\ge1} |\xi|^{-2(\gamma+\alpha)}\mu(\ud \xi)\int_{|r|^{\alpha_{2}}<|\xi|^{\alpha_{1}}} |r|^{1-2H} \ud r\\
=&~\int_{|r|\ge1} |r|^{-1-2\alpha_2} \ud r\int_{|\eta|\le1}\mu(\ud \eta)+\int_{|\xi|\ge1}|\xi|^{-2(\alpha+\gamma)H} \mu( \ud \xi)\int_{|\tau|\le1} |\tau|^{1-2H} \ud \tau\\
<&~\infty.
\end{align*}
Thus, we conclude that
$f_{2}(b, t, x, s,y)\le Cb^{-2}.$
The estimate of $\|v([0, a], t, x)-v([0, a], s, y)\|_2$ is similar to that of  $f_1$ and is omitted. The proof is complete.
\end{proof}

Let $B_r(t,x)$ be the open $\varrho$-ball in $\br_+\times \br^d$ of radius $ r$ centered at $(t, x)$. Now we study the regularity of the covariance function of $v$.
\begin{lemma}\label{corbound}
Assume that Assumption \ref{noise} holds and $\alpha_1\in(0,1).$ Let $B \subset\br_+\times \br^{d}$ be a compact box. Fix $(t, x) \in B$. Let $t^{\prime}=t-2(2 \delta)^{\alpha_{2}^{-1}}$ and $x^{\prime}=x$ (where $\delta$ is small enough so that $t^{\prime}>0$). Then there exists some constant $C:=C_{H,\gamma,\alpha,d,\delta}$ such that for all $\left(s_{1}, y\right),\left(s_{2}, z\right) \in$ $B_{2\delta}(t, x)$, 
\begin{align}\label{eq:Lemma52}
\left|\ee\left[(v(s_{1}, y)-v(s_{2}, z)) v(t^{\prime}, x^{\prime})\right] \right|\le C\big(\left|y-z\right|^{\delta_1}+\left|s_{1}-s_{2}\right|^{\delta_2}\big)
\end{align}
with $\delta_1\in(\alpha_1,1]$ and $\delta_2\in(\alpha_2,1]$.
\end{lemma}
\begin{proof}
For $(s, y) \in B_{2\delta}(t, x)$, define $f(s, y):=\ee[v(s, y) v\left(t^{\prime}, x^{\prime}\right)]$. Since $u$ and $v$ have the same law, we have
\begin{align*}
f(s, y)=&~\alpha_{H} \int_{0}^s\ud r\int_{0}^{t^{\prime}}\ud w|r-w|^{2H-2}\int_{\br^d}\mu(\ud \xi)e^{-(s-r)\Psi(\xi)}e^{-(t^{\prime}-w)\Psi(\xi)}e^{\mathrm{i}\langle\xi,y-x^{\prime}\rangle}\\
=&~\alpha_{H} \int_{0}^s\ud r\int_{0}^{t^{\prime}}\ud w|r-w|^{2H-2}\int_{\br^d}\mu(\ud \xi)e^{-(s-r)\Psi(\xi)}e^{-(t^{\prime}-w)\Psi(\xi)}\cos(\langle\xi,y-x^{\prime}\rangle).
\end{align*}
The Cauchy–Schwarz inequality and the Young inequality together with Lemma \ref{estofh} yield
\begin{align}\label{corest}
&~\alpha_{H} \int_{0}^{s} \ud r \int_{0}^{t^{\prime}}\ud w |r-w|^{2 H-2} e^{-(s-r) \Psi(\xi)} e^{-\left(t^{\prime}-w\right) \Psi(\xi)}\notag \\
\le&~\frac{1}{2} N_{s}(\xi)+\frac{1}{2} N_{t^{\prime}}(\xi)
\le C\left(\frac{1}{1+|\xi|^{2}}\right)^{(\alpha+\gamma) H}.
\end{align}

{\bf Step 1:} We prove \eqref{eq:Lemma52} for the case of $s_1=s_2$.

Using \eqref{corest} and the identity $\cos a-\cos b=-2\sin(\frac{a-b}{2})\sin(\frac{a+b}{2})$ for $a,b\in\mathbb R$ gives
\begin{align*}
&\quad\ |f(s, y)-f(s, z)|\\
&\le   C\int_{\mathbb{R}^{d}} \left(\frac{1}{1+|\xi|^{2}}\right)^{(\alpha+\gamma) H}|\cos (\langle\xi, y-x^{\prime}\rangle)-\cos (\langle\xi, z-x^{\prime}\rangle)| \mu(\ud \xi) \\
&\le C \int_{\mathbb{R}^{d}} \left(\frac{1}{1+|\xi|^2} \right)^{(\alpha+\gamma) H} \Big|\sin\big(\frac{1}{2} \langle\xi, y-z\rangle\big)\Big|\mu(\ud \xi).
\end{align*}
Furthermore, by using $|\sin a|\le |a|\wedge 1$ for $a\in \mathbb R$, we obtain
\begin{align*}
|f(s, y)-f(s, z)|
&\le C|y-z|\int_{|\xi| \le \frac{1}{|y-z|}} \left(\frac{1}{1+|\xi|^2} \right)^{(\alpha+\gamma) H}|\xi| \mu(\ud \xi)\\
&\quad~+C\int_{|\xi| > \frac{1}{|y-z|}} \left(\frac{1}{1+|\xi|^2} \right)^{(\alpha+\gamma) H} \mu(\ud \xi)\\
&=:J_1+J_2.
\end{align*}
Next, we estimate $J_1$ and $J_2$ separately. For $J_2$, the scaling property \eqref{scaling} implies
\begin{align*}
J_{2} & \le C \int_{|\xi| >\frac{1}{|y-z|}} |\xi|^{-2(\alpha+\gamma)H} \mu(\ud \xi)\le C |y-z|^{2\alpha_1}.
\end{align*}

To estimate $J_{1}$, we observe that $\left|y-z\right|^{-1} >\frac{1}{2}(2\delta)^{-\alpha_1^{-1}}:=c_{0}(\delta)$ holds for any $(s,y), (s,z) \in B_{2 \delta}(t, x)$. Thus, it holds that
\begin{align*}
J_{1} 
& \le C|y-z|\bigg(\int_{|\xi|<c_{0}(\delta)}\left|\xi\right|\mu(\ud \xi)+\int_{c_{0}(\delta) \le\left|\xi\right| \le \frac{1}{|y-z|}}|\xi|^{-2(\alpha+\gamma)H+1}\mu(\ud \xi)\bigg)\\
& \le C_\delta|y-z|\bigg(1+\int_{c_{0}(\delta)}^{\frac{1}{|y-z|}}r^{-2\alpha_1}\ud r\int_{\mathbb S^{d-1}}\Upsilon(z)\sigma(\ud z)\bigg),
\end{align*}
where the integral w.r.t. $r$ can be bounded as
\begin{align*}
\int_{c_{0}(\delta)}^{\frac{1}{|y-z|}}r^{-2\alpha_1}\ud r \le\begin{cases}
 C|y-z|^{1-2\alpha_1},&\text{if}\quad 2\alpha_1<1,\\
 \ln \frac{1}{c_0(\delta)|y-z|},&\text{if}\quad 2\alpha_1=1,\\
 C,&\text{if}\quad 2\alpha_1>1.
\end{cases}
\end{align*}
Hence, there is a $\delta_1\in(\alpha_1,1)$ such that
$$J_1\le C|y-z|^{\delta_1}.$$
The estimates of $J_1$ and $J_2$ yield that there exists $\delta_1\in(\alpha_1,1)$ such that $|f(s, y)-f(s, z)|\le C|y-z|^{\delta_1},$ which completes the proof of Step 1.

{\bf Step 2:} We prove \eqref{eq:Lemma52} for the case of $s_1\neq s_2$.

Assume without loss of generality that $s_1<s_2$. Then 
\begin{align*}
&~f(s_1, y)-f(s_2,y)\\
=\;&~\alpha_{H} \int_{0}^{s_1}\int_{0}^{t^{\prime}}|r-w|^{2H-2}\ud r\ud w\int_{\br^d}\mu(\ud \xi)\big[e^{-(s_1-r)\Psi(\xi)}-e^{-(s_2-r)\Psi(\xi)}\big]e^{-(t^{\prime}-w)\Psi(\xi)}\cos(\langle\xi,y-x^{\prime}\rangle)\\
\;&~-\alpha_{H} \int_{s_1}^{s_2}\int_{0}^{t^{\prime}}|r-w|^{2H-2}\ud r\ud w\int_{\br^d}\mu(\ud \xi)e^{-(s_2-r)\Psi(\xi)}e^{-(t^{\prime}-w)\Psi(\xi)}\cos(\langle\xi,y-x^{\prime}\rangle)\\
=:& ~I_1+I_2.
\end{align*}
Making use of \eqref{corest} and the inequality $1-e^{-\theta}\le1\wedge \theta$ for $\theta\ge0$, we deduce
\begin{align*}
\left|I_1\right| \le \;&~\alpha_{H} \int_{0}^{s_{1}} \int_{0}^{t^{\prime}}|r-w|^{2 H-2} \ud r \ud w \int_{\mathbb{R}^{d}} \mu(\mathrm{d} \xi) e^{-\left(s_{1}-r\right) \Psi(\xi)}\big[1-e^{-\left(s_{2}-s_{1}\right) \Psi(\xi)}\big] e^{-\left(t^{\prime}-w\right) \Psi(\xi)}\\
\le \;&~C \int_{\mathbb{R}^{d}}\big[1-e^{-\left(s_{2}-s_{1}\right) \Psi(\xi)}\big]\left(\frac{1}{1+|\xi|^{2}}\right)^{(\alpha+\gamma) H} \mu(\mathrm{d} \xi) \\
\le \;&~C\left|s_{2}-s_{1}\right| \int_{|\xi| \le\left|s_{2}-s_{1}\right|^{-\frac{1}{\alpha+\gamma}}}\Psi(\xi)\left(\frac{1}{1+|\xi|^{2}}\right)^{(\alpha+\gamma) H} \mu(\mathrm{d} \xi)\\
\;&~+C \int_{|\xi| >\left|s_{2}-s_{1}\right|^{-\frac{1}{\alpha+\gamma}}}\left(\frac{1}{1+|\xi|^{2}}\right)^{(\alpha+\gamma) H} \mu(\mathrm{d} \xi).
\end{align*}
Further, applying the polar coordinate transform and Lemma \ref{Upsilon} (iii) yields
\begin{align}\label{eq:I12}\notag
I_{1,2}:=& ~\int_{|\xi| >|s_{2}-s_{1}|^{-\frac{1}{\alpha+\gamma}}}\left(\frac{1}{1+|\xi|^{2}}\right)^{(\alpha+\gamma) H} \mu(\mathrm{d} \xi) \\\notag
=&~\int_{|s_{2}-s_{1}|^{-\frac{1}{\alpha+\gamma}}}^\infty\left(\frac{1}{1+\rho^{2}}\right)^{(\alpha+\gamma) H} \rho^{\beta-1}\ud\rho\int_{\mathbb S^{d-1}}\Upsilon(z)
\sigma(\mathrm{d} z)\\
\le &~C|s_1-s_2|^{2\alpha_2},
\end{align}
and similarly,
\begin{align*}
\tilde I_{1,1}:=&~ |s_1-s_2|\int_{|\xi| \le|s_{2}-s_{1}|^{-\frac{1}{\alpha+\gamma}}}|\xi|^{\alpha+\gamma}\left(\frac{1}{1+|\xi|^{2}}\right)^{(\alpha+\gamma) H} \mu(\mathrm{d} \xi)\\
\le&~ |s_1-s_2|\int_0^{|s_{2}-s_{1}|^{-\frac{1}{\alpha+\gamma}}}\left(\frac{1}{1+\rho^{2}}\right)^{(\alpha+\gamma) H} \rho^{\alpha+\gamma+\beta-1}\ud\rho.
\end{align*}
Since $|s_1-s_2|^{-\frac{1}{\alpha+\gamma}}> \frac{1}{2}(2\delta)^{-\frac{1}{\alpha_2(\alpha+\gamma)}}=:c_1(\delta)$ for any $(s_1,y)$, $(s_2,y)\in B_{2\delta}(t,x)$, we have
\begin{align*}
&\quad\ \int_0^{|s_{2}-s_{1}|^{-\frac{1}{\alpha+\gamma}}}\left(\frac{1}{1+\rho^{2}}\right)^{(\alpha+\gamma) H} \rho^{\alpha+\gamma+\beta-1}\ud\rho\\
&\le C\int_0^{c_1(\delta)}\rho^{\alpha+\gamma+\beta-1}\ud\rho+C \int_{c_1(\delta)}^{|s_{2}-s_{1}|^{-\frac{1}{\alpha+\gamma}}}\rho^{\alpha+\gamma-2\alpha_1-1}\ud\rho,
\end{align*}
where the first integral is bounded and the second integral can be bounded as
\begin{align*}
 \int_{c_1(\delta)}^{|s_{2}-s_{1}|^{-\frac{1}{\alpha+\gamma}}}\rho^{\alpha+\gamma-2\alpha_1-1}\ud\rho\le\begin{cases}
 C|s_1-s_2|^{2\alpha_2-1},&\text{if}\quad\alpha+\gamma-2\alpha_1>0,\\
 -\frac{1}{\alpha+\gamma}\ln |s_1-s_2|-\ln c_1(\delta),&\text{if}\quad\alpha+\gamma-2\alpha_1=0,\\
 C,&\text{if}\quad\alpha+\gamma-2\alpha_1<0.
\end{cases}
\end{align*}
Gathering the above estimates, we conclude
$$
\tilde I_{1,1}\le C_{\varepsilon}|s_1-s_2|^{2\alpha_2\wedge (1-\varepsilon\mathbb 1_{\{2\alpha_1=\alpha+\gamma\}})}.
$$
Here and after, $\varepsilon\in(0,1-\alpha_1)$ is a sufficiently small constant.
Consequently, the relation $\Psi(\xi)\le (1+|\xi|^2)^{\frac{\alpha+\gamma}{2}}\le C(1+|\xi|^{\alpha+\gamma})$ and Assumption \ref{dalang-condition} indicate 
\begin{align} \label{eq:I11}\notag
I_{1,1}:=&~C|s_{2}-s_{1}| \int_{|\xi| \le\left|s_{2}-s_{1}\right|^{-\frac{1}{\alpha+\gamma}}}\Psi(\xi)\left(\frac{1}{1+|\xi|^{2}}\right)^{(\alpha+\gamma) H} \mu(\mathrm{d} \xi)\\\notag
\le &~C|s_2-s_1|\int_{\br^d}\left(\frac{1}{1+|\xi|^{2}}\right)^{(\alpha+\gamma) H} \mu(\mathrm{d} \xi)+\tilde I_{1,1} \\
\le&~C_{\varepsilon}|s_1-s_2|^{2\alpha_2\wedge (1-\varepsilon\mathbb 1_{\{2\alpha_1=\alpha+\gamma\}})}.
\end{align}
Thus, we conclude that 
\begin{align}\label{eq:I1}
\left|I_1\right| \le C(I_{1,1}+I_{1,2})\le C|s_1-s_2|^{2\alpha_2\wedge (1-\varepsilon\mathbb 1_{\{2\alpha_1=\alpha+\gamma\}})}.
\end{align}

It remains to estimate $I_2$. Since $|s-t| \le(2 \delta)^{\alpha_2^{-1}}$ holds 
for all $(s,y) \in B_{2 \delta}(t, x)$, we get
$$s-t^{\prime} \ge t-t^{\prime}-|s-t| \ge(2 \delta)^{\alpha_2^{-1}}=: c_{2}(\delta)>0,$$
which indicates that
\begin{align*}
|I_2|  \le \;&\,\alpha_{H} \int_{s_{1}}^{s_{2}} \int_{0}^{t^{\prime}} c_{2}(\delta)^{2 H-2} \ud r \ud w \int_{\mathbb{R}^{d}} \mu(\mathrm{d} \xi) e^{-(s_2-r)\Psi(\xi)} e^{-\left(t^{\prime}-w\right) \Psi(\xi)} \\
\le\; &\,C \int_{\br^d} \frac{1-e^{-(s_2-s_1)\Psi(\xi)}}{\Psi(\xi)} \frac{1-e^{-t^{\prime} \Psi(\xi)}}{\Psi(\xi)} \mu(\mathrm{d} \xi)\\
\le\;&\,C |s_2-s_1|\int_{|\xi|<|s_2-s_1|^{-\frac{1}{\alpha+\gamma}}} \frac{1-e^{-t^{\prime} \Psi(\xi)}}{\Psi(\xi)} \mu(\mathrm{d} \xi)+C \int_{|\xi|\ge|s_2-s_1|^{-\frac{1}{\alpha+\gamma}}} \frac{1}{\Psi(\xi)^2} \mu(\mathrm{d} \xi)\\
=:&\;I_{2,1}+I_{2,2}.
\end{align*}
For each $\vartheta>0$, there exists some constant $C:=C_{\vartheta}$ such that 
\begin{equation}\label{eq:lower}
\Psi(\xi)^2\ge C(1+|\xi|^2)^{\alpha+\gamma}\ge C(1+|\xi|^2)^{(\alpha+\gamma)H}\quad\quad\forall~|\xi|\ge \vartheta .
\end{equation}
By virtue of \eqref{eq:lower} and
 \eqref{muhyp}, 
\begin{align*}
&~\int_{|\xi|<|s_2-s_1|^{-\frac{1}{\alpha+\gamma}}} \frac{1-e^{-t^{\prime} \Psi(\xi)}}{\Psi(\xi)} \mu(\mathrm{d} \xi)\\
=&~\int_{|\xi|<c_1(\delta)} \frac{1-e^{-t^{\prime} \Psi(\xi)}}{\Psi(\xi)} \mu(\mathrm{d} \xi)+\int_{c_1(\delta)\le|\xi|<|s_2-s_1|^{-\frac{1}{\alpha+\gamma}}} \frac{1-e^{-t^{\prime} \Psi(\xi)}}{\Psi(\xi)} \mu(\mathrm{d} \xi)\\
\le&~t^{\prime}\int_{|\xi|<c_1(\delta)} \mu(\mathrm{d} \xi)+\int_{c_1(\delta)\le|\xi|<|s_2-s_1|^{-\frac{1}{\alpha+\gamma}}}\Psi(\xi) \frac{1}{\Psi(\xi)^2} \mu(\mathrm{d} \xi)\\
\le&~C+C\int_{|\xi|<|s_2-s_1|^{-\frac{1}{\alpha+\gamma}}}\Psi(\xi) \left(\frac{1}{1+|\xi|^2}\right)^{(\alpha+\gamma)H} \mu(\mathrm{d} \xi),
\end{align*}
which together with \eqref{eq:I11} gives
$$
I_{2,1}\le  C|s_{2}-s_{1}|+CI_{1,1}\le C_{\delta, H, d,\varepsilon}|s_{2}-s_{1}|^{2\alpha_2\wedge(1-\varepsilon\mathbb 1_{\{2\alpha_1=\alpha+\gamma\}})}.
$$
On the other hand, since $|s_2-s_1|^{-\frac{1}{\alpha+\gamma}}>c_1(\delta)$, we apply \eqref{eq:lower} and \eqref{eq:I12} to obtain
\begin{align*}
I_{2,2} =&~C\int_{|\xi|\ge|s_2-s_1|^{-\frac{1}{\alpha+\gamma}}} \frac{1}{\Psi(\xi)^2} \mu(\mathrm{d} \xi)\\
\le &~C_\delta\int_{|\xi|\ge|s_2-s_1|^{-\frac{1}{\alpha+\gamma}}} \left(\frac{1}{1+|\xi|^2} \right)^{(\alpha+\gamma)H}\mu(\mathrm{d} \xi)\\
=&~C_\delta I_{1,2}\le C|s_1-s_2|^{2\alpha_2}.
\end{align*}
Hence, combining the estimates of $I_{2,1}$ and $I_{2,2}$, we arrive at
$$
\left|I_{2}\right| \le C_{\delta, H, \beta,d,\varepsilon}|s_{2}-s_{1}|^{2\alpha_2\wedge(1-\varepsilon\mathbb 1_{\{2\alpha_1=\alpha+\gamma\}})},
$$
which in combination with \eqref{eq:I1} completes the proof.
\end{proof}

\begin{theorem}\label{critical}
Assume that Assumption \ref{noise} holds and $\alpha_1\in(0,1).$ Then $Q$ is the critical dimension for hitting points and points are polar for $\hat{u}$, that is, for all $z \in \br^Q$,
$$
\bp\{\exists(t, x) \in( 0,+\infty)\times \br^d: \hat{u}(t, x)=z\}=0.
$$
\end{theorem}
\begin{proof}
By Corollary \ref{polaritypoints}, $Q=\frac{d}{\alpha_1}+\frac{1}{\alpha_2}$ is the critical dimension for hitting points. Let $\hat{v}=\left\{\hat{v}(t, x)\right\}_{(t,x) \in \br_+\times \br^d}$ be an $\br^Q$-valued random field whose components $\hat{v}_j, j = 1,\dots,Q$, are independent copies of $v$.
Based on \eqref{boundsigma}, Lemma \ref{approximationfield} and Lemma \ref{corbound}, it follows from \cite[Theorem 2.6]{MR4317713} that for all $z \in \br^Q$ and all compact rectangles $I\times J \subset(0, \infty)\times \br^d$,
$$
\bp\{\exists(t, x) \in I\times J: \hat{u}(t, x)=z\}=\bp\{\exists(t, x) \in I\times J: \hat v(t, x)=z\}=0,
$$
where we used the fact that $u$ and $v$ has the same law. Since this holds for all compact rectangles $I\times J \subset(0, \infty)\times \br^d$, the proof is finished.
\end{proof}

\begin{remark}\label{conjection}
For the linear stochastic biharmonic heat equation on a $d$–dimensional torus $(d = 1,2,3)$ driven by space-time white noise which corresponds to \eqref{FKE} with $\alpha=0,\gamma=2,H=\frac{1}{2}$ and $\beta=d$, the issue of polarity for singletons in the critical dimension was proposed by the authors of \cite{Hinojosa} as an open question for further investigations. And they guessed that the singletons are polar in the critical dimension.  We mention that all the proofs in Subsection \ref{points-critical} are valid for the case of $H=\frac{1}{2}$ and thus provide strong evidence for this conjecture.
\end{remark}

\bibliographystyle{amsplain}
\bibliography{reference}

%\ACKNO{}

\end{document}